\documentclass[12pt]{article}
\usepackage{graphicx}
\usepackage[centering, margin={0.8in, 0.5in}, includeheadfoot]{geometry}

\usepackage[utf8x]{inputenc}
\usepackage{amssymb,pdfsync,epstopdf,verbatim,gensymb,amsmath,amsthm}
\usepackage{nccmath, mathtools}
\usepackage{textcomp}
\usepackage{amsmath}
\usepackage{float}
\usepackage{subfig,boldline}
\usepackage{amsfonts}
\usepackage{epsfig,color, hyperref}

\def\R {{\mathbb R}}

\def\H01{{H_0^1(\Omega)}}
\def\L2{{L^2(\Omega)}}

\usepackage{bmpsize}

\newtheorem{definition}{Definition}[section]
\newtheorem{lemma}{Lemma}[section]

\newtheorem{remark}{Remark}[section]

\newtheorem{proposition}{Proposition}
\newtheorem{algorithm}{Algorithm}[section]
\DeclareMathOperator*{\argmin}{arg\,min}
\DeclareMathOperator*{\sgn}{sign}

\newcommand{\wh}{\widehat}

\newcommand{\Beq}{\begin{equation}}
\newcommand{\Eeq}{\end{equation}}
\newcommand{\beq}{\begin{equation*}}
\newcommand{\eeq}{\end{equation*}}
\newcommand{\bal}{\begin{align}}
\newcommand{\eal}{\end{align}}

\renewcommand{\L}{\langle}

\newcommand{\bp}{\begin{prob}}
\newcommand{\ep}{\end{prob}}
\newcommand{\bpr}{\begin{proof}}
\newcommand{\epr}{\end{proof}}

\newcommand{\bel}[1]{\begin{equation}\label{#1}}
\newcommand{\ee}{\end{equation}}

\graphicspath{{./figs/}}

\author{
Madhu Gupta{\footnote{
madhu.gupta@mavs.uta.edu, Department of Mathematics, University of Texas at Arlington, 655 W. Mitchell Street, 222H SEIR Building, Arlington, Texas-76010, USA.}}\and
Rohit Kumar Mishra{\footnote{
rohit.mishra@uta.edu, Department of Mathematics, University of Texas at Arlington, 655 W. Mitchell Street, 222B SEIR Building, Arlington, Texas-76010, USA}}\and
Souvik Roy{\footnote{
souvik.roy@uta.edu, Department of Mathematics, University of Texas at Arlington, 655 W. Mitchell Street, 219 SEIR Building, Arlington, Texas-76010, USA; Tel: +1 (817) 272-5748 (Corresponding author).}}
}

\date{}

\begin{document}

\title{Sparse reconstruction of log-conductivity in current density impedance tomography}

\maketitle
\begin{abstract}
A new non-linear optimization approach is proposed for the sparse reconstruction of log-conductivities in current density impedance imaging. This framework comprises of minimizing an objective functional involving a least squares fit of the interior electric field data corresponding to two boundary voltage measurements, where the conductivity and the electric potential are related through an elliptic PDE arising in electrical impedance tomography. Further, the objective functional consists of a $L^1$ regularization term that promotes sparsity patterns in the conductivity and a Perona-Malik anisotropic diffusion term that enhances the edges to facilitate high contrast and resolution. This framework is motivated by a similar recent approach to solve an inverse problem in acousto-electric tomography. Several numerical experiments and comparison with an existing method demonstrate the effectiveness of the proposed method for superior image reconstructions of a wide-variety of log-conductivity patterns. 
\end{abstract}

Keywords: {Inverse problems, PDE-constrained optimization, proximal methods, edge-enhancement, sparsity patterns, current density impedance imaging. }\\

MSC: {35R30, 49J20, 49K20,  65M08, 82C31}

\section{Introduction}
Electrical impedance tomography (EIT) is an imaging modality, where one attempts to recover the conductivity of a body from the boundary measurement of current and voltage \cite{cheney1999}. The underlying inverse problem is highly ill-posed and non-linear yet very important due to its wide range applications in the fields such as medical imaging \cite{Zain} and engineering \cite{Kruger,Waterfall}. The mathematical formulation of the EIT inverse problem is given by the following conductivity equation
\begin{equation}\label{elliptic_BVP}
\begin{aligned}
-\nabla\cdot(\sigma(x)\nabla{u(x)})&=0 \quad x \in \Omega, \\
\sigma(x)\dfrac{\partial u}{\partial \nu}(x)&=f(x) x \in \Gamma,\\
\end{aligned}
\end{equation}
where $\Omega\subset\mathbb{R}^n$ is a convex and bounded domain with Lipschitz boundary and $\Gamma$ is the boundary of $\Omega$. In this model, $\sigma$ is the electrical conductivity, $u$ represents the electric potential and $f$ is the current applied to the boundary. 
 
The reconstructions obtained through the EIT setup usually have high contrast but limited spatial resolution \cite{seagar1987}. On the other hand, reconstructions obtained through ultrasound imaging has very high resolution but limited contrast \cite{Gaik2018,Roy2015}. In recent years, attempts have been made to combine multiple imaging modalities to obtain image reconstructions with both high contrast and high resolution. This led to the emergence of hybrid imaging methods that belong to class of coupled-physics imaging modalities to generate images of superior quality. One of such imaging methods, known as current density impedance imaging  (CDII) combines the classical EIT setup with magnetic resonance (MR) scanning \cite{Has08,Na11}. It is alternatively known as magnetic resonance EIT (MREIT). Current or voltage is applied through the electrodes, which give rise to an interior electric field and the corresponding generated magnetic field, represented as $B=(B_x,B_y,B_z)$, is measured by the MR scanner. The corresponding inverse problem is to solve for the conductivity $\sigma$ from $B_z$ using the well-known iterative Harmonic $B_z$-algorithm \cite{Ga18,Liu07}. Convergence of the harmonic $B_z$ algorithm has been well-studied \cite{Ga18,Liu07,Liu10}. In particular, it has been shown that for small contrast values of the target conductivity, the harmonic $B_z$-algorithm is stable and convergent, provided we have a good initial guess \cite{Liu07}. Thus, it is not clear that one can recover good quality images for high contrast objects through Harmonic $B_z$-algorithm.
 
An alternate approach to solve the CDII inverse problem is to use the knowledge of interior electric field, which is obtained from the magnetic field. Correspondingly, the magnitude of the interior electric field is also determined \cite{Sc91,Se11}, which is given by
\begin{equation}\label{interior_data}
H(\sigma(x))=\sigma(x) |\nabla u(x)|, \quad x\in\Omega. 
\end{equation}
The formulation of reconstruction problem is as follows: Given the boundary data $f$ for, possibly, several choices of boundary patterns and the corresponding interior measurement data $H$, find the conductivity distribution $\sigma$. In this framework, we use the internal function $H(\sigma)$ to replace $\sigma$ in the EIT equation \eqref{elliptic_BVP} to get the following nonlinear equation
\begin{equation}\label{1_Laplacian}
\begin{aligned}
\nabla\cdot\Bigg(\dfrac{H}{|\nabla u|}\nabla{u}\Bigg)=0 
& \mbox{ in }\Omega, \\
\dfrac{H}{|\nabla u|}\dfrac{\partial u}{\partial \nu}=f&~\mbox{on } {\Gamma}.
\end{aligned}
\end{equation}
For the CDII inverse problem, the solution to the boundary value problem (\ref{1_Laplacian}) is crucial but it is difficult to use it in practice because of its highly nonlinear behaviour and also because the data represented by the measured values of $H$ enter as a coefficient of the differential model \cite{seagar1987}. Even with the additional measurements, analysis and application of the 1-Laplacian relies on an iterative localized algorithm, wherein one considers an approximation of the CDII problem. This subsequently led to several computational approaches in solving the CDII inverse problem. In \cite{kuchment}, it was proved that the linearized problem is elliptic and hence solvable, if there are at least $n$ set of measurements $\{H_i(\sigma)\}_{i=1}^n$  and corresponding to $n$ boundary  data $\{f_i\}_{i=1}^n$ such that $\nabla u_i$ and $\nabla u_j$ are nowhere collinear for $i \neq j$. It has been shown in \cite{Kim2002} that the solution of the above 1-Laplacian equation with the Neumann boundary condition is non-existent unless additional measurements with different boundary current patterns are used. Recovery of isotropic conductivity in regions where the magnetic field is transversal using two internal current distributions was done using an explicit local formula \cite{Lee04}. Moreover, using the information of two internal current distributions, the authors in \cite{Kim2002} uniquely determine the singular support of the conductivity function. In \cite{Na07}, the authors showed that the conductivity in the planer domain can be recovered  from a single voltage-current on a part of boundary and the magnitude of one interior current density. In the same article, they also provide  sufficient conditions on Dirichlet boundary data to guarantee unique recovery of conductivity. In \cite{Na09}, the recovery of H\"{o}lder continuous conductivities have been establised for domains with connected boundary from the interior measurement of the magnitude of one current density. Determination of  isotropic conductivity variations from measurements of two current density vector fields was studied in \cite{Has08}.  In \cite{Tamasan}, authors showed the recovery of  planar conductivities by solving the 1-Laplace equation with partial boundary data. The authors in \cite{Monard2012,Bal2014,Bal_22014} present explicit reconstruction formulae for recovering the conductivity distribution from multiple interior measurements in two and higher dimensions.

The well-known numerical reconstruction algorithm using the internal current distribution is an iterative $J$-substitution algorithm which was first introduced by \cite{Kwon02} and subsequently considered in other works, see for e.g., in \cite{Khang02,Kim03,Na09,Na11}. It has been shown that the $J$-substitution algorithm is able to reconstruct the conductivity with high resolution \cite{Kim03,Kim2002}. Another numerical reconstruction iterative method is the regularized D-bar method \cite{Kn09} that provides images with high resolution. In \cite{Mora12}, the authors use an alternating split Bregman algorithm for solving a minimization problem related to the energy functional corresponding to the 1-Laplacian equation (\ref{1_Laplacian}). Also, in \cite{Hoff2014}, Picard and Newton type algorithms are implemented to solve the 1-Laplacian problem. But there is not enough evidence to suggest that these existing algorithms (linearized or localized iterative methods) can provide high contrast images, specially for objects with holes or inclusions, which are inherent to CDII reconstructions. 

The CDII inversion problem can be viewed as an inverse problem of estimating the conductivity parameter from the conductivity partial differential equation (PDE) (\ref{elliptic_BVP}). A robust way of solving this inverse problem is to formulate it as an optimal control problem, where the condiuctivity parameter is the control variable that drives the interior electric field close to the measured value, with dynamics governed by the conductivity PDE. Such optimal control methods has been used previously in the context of ultrasonically-induced Lorentz force electrical impedance tomography \cite{Ammari2015}, magnetoacoustic tomography \cite{Ammari_22015, Ammari2017,Qiu2015} and acousto-electric tomography (AET) \cite{Ades2018,Roy_AET2}. In \cite{Ammari2015}, the authors use an optimal control framework to recover the conductivity distribution from the measurements of current induced by static magnetic field through the Lorentz force. They solve the optimal control problem using an orthogonal field method.  In \cite{Ammari_22015, Ammari2017,Qiu2015}, the authors use an optimal control approach to reconstruct conductivity distribution of biological tissue from measurements of the Lorentz force induced tissue vibration. In \cite{Ades2018,Roy_AET2}, the authors reconstruct log-conductivity in acousto-electric tomography (AET) using an optimal control formulation based on the theory of anisotropic diffusion to potentially obtain reconstructions with high resolution and contrast. The results obtained in these papers demonstrated that such optimal control frameworks are robust and accurate for imaging modalities arising through a partial differential equation (PDE). 

In this paper, we consider a similar optimization framework developed in \cite{Roy_AET2} for reconstructing the log-conductivity in CDII. We formulate a minimization problem, where given interior electric field intensity data, we aim at determining the variation in conductivity from a known background conductivity. We, further, assume that this variation demonstrates a sparsity pattern. Such patterns arise frequently in several tomographic imaging scenarios, for e.g. in blood vessel tomographic reconstructions \cite{Li2002}. This is incorporated in our model through a $L^2-L^1$ regularization term in our objective functional. To obtain sharp edges and, thus, improve spatial resolution of the reconstructed images, we use a Perona-Malik anisotropic diffusion filtering term in our functional. The resulting optimality system gives rise to an elliptic adjoint equation with a $L^2$ source term. Classical cell-nodal finite difference schemes are not applicable for solving such equations. We, thus, use a averaged cell-nodal scheme to solve such equations. Finally, we solve the optimization problem using a variable inertial proximal scheme that efficiently handles the non-differentiable terms in the objective functional. We demonstrate through several examples that our method can be used to obtain superior quality reconstructions for objects with holes and inclusions. In this context we would like to remark that our framework can also be used for obtaining superior reconstruction of anisotropic conductivity distributions, in ultrasonically-induced Lorentz force electrical impedance tomography and magnetoacoustic tomography. Further, our proposed framework can also be used in other hybrid imaging modalities like quantitative photoacoustic tomography \cite{Gao12} and quantitative thermoacoustic tomography \cite{Ammari2013} to obtain better reconstructions.

The article is organized as follows: In the Section \ref{sec:Minimization_problem}, we formulate the minimization problem for the CDII.  In the Section \ref{sec:Theory}, we present some theoretical results about our optimization problem. We also characterize the optimality system. The variable inertial proximal scheme and the averaged cell-nodal schemes to solve the optimization problem are discussed in Section \ref{sec:numsol}. In the Section \ref{sec:numexp}, we present simulation results of our CDII framework and compare them with the reconstructions obtained using the Picard scheme proposed in \cite{Hoff2014}, which validate our framework for CDII and demonstrate the effectiveness of our method to reconstruct wide variety of objects with corners, holes and inclusions. A section on conclusions completes our work.

\section{A minimization problem}\label{sec:Minimization_problem}
We consider the conductivity equation in $\R^2$ arising in EIT
\begin{equation}\label{eq1}
\begin{aligned}
-&\nabla\cdot(e^{\sigma(x,y)}\nabla{u(x,y)})=0~ \mbox{in }\Omega, \\
&u(x,y)|_{\Gamma}=f_D(x,y),\\
\end{aligned}
\end{equation}
where $\Omega\subset \mathbb{R}^2$ is bounded, $\Gamma$ is the boundary of $\Omega$, $e^\sigma$ is the conductivity coefficient and $u\in H^1_{f_D}(\Omega)=\lbrace u\in H^1(\Omega) : u=f_D \mbox{ on } \Gamma \rbrace$ is the electric potential. 

We assume that $\sigma$ is a sparse conductivity coefficient which we want to recover, given the fact that the conductivity of the background is 1. The conductivity equation (\ref{eq1}) can also be written as 
$$
\mathcal{L}(u,\sigma,f_D)=0,
$$
where 
$\sigma(x,y)\in L_{ad} = \lbrace \sigma \in H_0^1(\Omega): \sigma_l\leq\sigma(x,y) \leq \sigma_u,~ \forall (x,y)\in\Omega \rbrace$, $\sigma_u >0$ and $\sigma_l < 0$ .

We consider an optimization-based approach for reconstructing $\sigma$ given $H_1(\sigma),H_2(\sigma)$, where 
\[
H(\sigma) = e^\sigma |\nabla u|
\]
is the interior electric field corresponding to the voltage potential $u$. We consider the following cost functional 
\begin{equation}\label{functional}
\begin{aligned}
J(\sigma,u_1,u_2)= &\sum_{j=1}^2~\dfrac{\alpha_j}{2}\int_\Omega~  (H_j(x,y)-H_j^\delta(x,y))^2~dxdy  +\dfrac{\beta}{2} \|\sigma\|^2_{L^2(\Omega)} \\
&+ \gamma \, \|\sigma\|_{L^1(\Omega)}+\dfrac{\delta}{2}\int_\Omega \log (1+|\nabla \sigma(x,y)|^2)~dx dy
\end{aligned}
\end{equation}
where $u_1,u_2$ satisfy (\ref{eq1}) with boundary data $f_D^1,f_D^2$.
We now consider the following minimization problem
\begin{equation}\label{min_problem}
\begin{aligned}
\min_\sigma & ~J(\sigma, u_1, u_2),\\
\mbox{ s.t. } &\mathcal{L}(u_1,\sigma,f_D^1)=0,\\
& \mathcal{L}(u_2,\sigma,f_D^2)=0.
\end{aligned} \tag{P}
\end{equation}

The term $\gamma \, \|\sigma\|_{L^1(\Omega)}$, $\gamma > 0$  in the functional, defined in (\ref{functional}), implements a $L^1$ regularization of the minimization problem that promotes sparsity patterns in the reconstruction of conductivity. Such a regularization method mirrors the well known compressed-sensing technique; see  \cite{Candes}. In recent past, optimal control with $L^1$ cost functionals has become a topic of major interest \cite{stadler}, because one obtains sparse controls through this procedure, which finds numerous applications. The motivation for sparse log-conductivity patterns is based on the assumption that the background conductivity is known to be 1 in a substantial part of the domain $\Omega$ after normalization and varies considerably from this value in correspondence to different kind of objects present within the domain.

The combined $L^2$-$L^1$ regularization allows for the reconstruction of conductivity, and thus the imaging of, possibly, irregular objects inside $\Omega$. This does not serve the ultimate goal of reconstructing objects like tissues in medical imaging, which are more regular, save for the edges that eventually define them. We infuse this additional aprior knowledge into our model through the last term in our functional (\ref{functional}) that, commonly, appears in the field of anisotropic diffusion. Such a term plays an important role in 
dampening image noise while keeping significant parts of the image content such as edges and other anatomical details that are of utmost importance in the interpretation of the image. Anisotropic diffusion means non-uniform diffusion in different directions. The regions where $|\nabla\sigma|$ is very small corresponds to noise and thus, the process of smoothening occurs. At the edges or singularities of an object, where the value of $|\nabla\sigma|$ is large, there is a small amount of smoothening and this preserves the edges. A standard technique to implement anisotropic diffusion, in order to obtain a good contrast, is to use a total variation (TV) regularization \cite{Chan_Esedoglu_Nikolova_2006, RudinOsherFatemi1992}. But, this regularization method gives rise to a non-differentiable term in the functional (\ref{functional}), thus requiring more sophisticated optimization algorithms. On the other hand, anisotropic diffusion is inherent to Perona-Malik (PM) filtering \cite{Perona_Malik}. It is well-known that the diffusion process governed by the PM equation leads to a decrease in the total variation during its evolution \cite{ScherzerWeickert2000}. We, thus, choose the energy functional of the Perona-Malik equation for anisotropic diffusion \cite{Perona_Malik}. One can note that the PM regularization term is differentiable and, thus, easier to handle than the TV regularization term.

Mathematically, one can consider the PM filtering as the gradient flow generated by the non-convex and lower semi-continuous functional given by
$$
J_{PM}(\sigma) = \int_\Omega \log (1+|\nabla \sigma(x,y)|^2)~dxdy.
$$  
We refer to \cite{Catte1992,ScherzerWeickert2000} for a general introduction to anisotropic diffusion and a detailed discussion on the PM functional. Further, in \cite{Roy_AET2}, the PM model was used in the reconstruction of log-conductivities in AET and it was observed that the reconstructions obtained demonstrated superior contrast and resolution. Thus, for the current setup in CDII, we use a similar PM anisotropic diffusion filter to facilitate high contrast and high resolution images.

\section{Theory of the minimization problem}\label{sec:Theory}
In this section, we discuss the existence of solutions of the minimization problem (\ref{min_problem}) and its characterization through a first-order optimality system. We refer to this minimization problem as the CDII sparse reconstruction problem (CDII-SR). Our analysis of this problem begins with the discussion concerning the existence and uniqueness of weak solutions of $\mathcal{L}(u,\sigma, f_D) = 0$, which can be proved by standard arguments of Riesz representation theorem \cite[Chapter 8]{Gilbarg_Trudinger_2001}.
\begin{proposition}
Let $\sigma \in L_{ad}$ and $f_D \in H^{1/2}(\partial\Omega)$. Then the problem \eqref{eq1} has a unique solution in $H^1_{f_D}(\Omega)$.   
\end{proposition}
The solvability of the CDII inversion problem depends on the type of Dirichlet boundary data  $f_D^j,~j=1,2$. In this, context, we 
have the following lemma from \cite{alessandrini_nesi}.

\begin{lemma}[Boundary data]\label{boundary_conditions}
Let $\Omega \subset \R^2$ be a bounded simply connected open set, whose boundary $\Gamma$ is a simple closed curve. 
Let $f = (f^1,f^2)$ be a mapping $\Gamma \to \R^2$ which is a homeomorphism of $\Gamma$ onto a convex closed curve $C$, and let $D$ denote the bounded convex domain bounded by $C$. 
Let $\sigma \in L^\infty (\Omega)$, and let $U = (u_1,u_2)$ be the $e^\sigma$-harmonic mapping whose components 
$u_1$ and $u_2$ are solutions to the Dirichlet problem 
\eqref{eq1} with $f_D=f_D^1$ and $f_D=f_D^2$, 
respectively, and $f_D^J \in H^1(\Omega) \cap C(\bar \Omega)$, $J = 1,2$. Then $U$ is a homeomorphism of $\Omega$ onto $D$. In particular, for all $\omega \subset\subset \Omega $ we have either $\det (\nabla u_1, \nabla u_1 ) > 0$ or 
 $\det (\nabla u_1, \nabla u_1 ) < 0$ almost everywhere 
 in $\omega$. 
 \end{lemma}
In \cite{kuchment}, the authors have shown that in 2D, the boundary condition pair $f_D^1=x$ and $f_D^2=y$ satisfies the conditions of Lemma \ref{boundary_conditions} and, thus, the corresponding solutions to (\ref{eq1}) $u_1$ and $u_2$ have no critical points and $\nabla u_1,\nabla u_2$ are not collinear in $\bar{\Omega}$. We will use these boundary conditions for our numerical experiments in Section \ref{sec:numexp}.

Next, we consider the  Fr\'{e}chet differentiability of the  mapping $u(\sigma)$.
\begin{lemma}\label{differentiable_constraint}
The map $u(\sigma)$ defined by (\ref{eq1}) is Fr\'{e}chet differentiable as a mapping from $L_{ad}$ to $H^1_{f_D}(\Omega)$.
\end{lemma}
For the proof of this Lemma, we refer to \cite{kuchment}. Using Lemma \ref{differentiable_constraint}, we introduce the reduced cost functional
\begin{equation}\label{reduced_func}
\wh{J}(\sigma) = J(\sigma, u_1(\sigma), u_2(\sigma)),
\end{equation}
where $u_i(\sigma)$, $i=1,2$ denotes the unique solution of \eqref{eq1} given $\sigma$ and $f_D^i,i=1,2$. The constrained optimization problem (\ref{min_problem}) can be formulated as an unconstrained one as follows
\begin{equation}\label{reduced_min}
\min_{\sigma\in L_{ad}} \hat{J}(\sigma).
\end{equation}We next investigate the existence of a minimizer to the CDII-SR problem (\ref{min_problem}). We first consider the case when $\delta=0$, i.e., the Perona-Malik term in the functional $J$ is absent. 
\begin{proposition}
Let $f_D^1,f_D^2 \in H^{1/2}(\Omega)$ such that $|\nabla u_1|>0,|\nabla u_2|>0$ and let $\delta=0$. Then there exists a triplet $(\sigma^*,u_1^*,u_2^*) \in L_{ad}\times H^1_{f_D^1}(\Omega)\times H^1_{f_D^2}(\Omega)$ such that $u_i^*, i=1,2$ are solutions to $\mathcal{L}(\sigma,u_i,f^i_D)=0, i=1,2$ and $\sigma^*$ minimizes $\hat{J}$ in $L_{ad}$.
\end{proposition}
\begin{proof}
Boundedness from below of $\hat{J}$ guarantees the existence of a minimizing sequence $(\sigma^m)$. Since $L_{ad}$ is reflexive and $\hat{J}$ is sequentially weakly lower semi-continuous,
this sequence is bounded. Therefore it contains a weakly convergent 
subsequence $(\sigma^{m_l})$ in $L_{ad}$, $\sigma^{m_l} \rightharpoonup \sigma^*$. 
Correspondingly, the sequence $(u_1^{m_l},u_2^{m_l})$, where $u_i^{m_l}=u_i(\sigma^{m_l})$, is bounded in $ H^1_{f_D^1}(\Omega)\times H^1_{f_D^2}(\Omega) $. Therefore the sequence 
converges weakly to $(u_1^*,u_2^*)$. Now, using the Rellich Kondrachev compactness theorem in $\mathbb{R}^2$, we have that $L_{ad}$ is compactly embedded in $L^2(\Omega)$. This results in a strong convergence of the subsequence $\sigma^{m_l}$ in $ L^2(\Omega)$ to $\sigma^*$. We, now, consider the weak formulation of the solutions of the elliptic problem (\ref{eq1}) and, thus, focus on $\langle \nabla \cdot (\sigma^{m_l}\nabla u_i^{m_l}), \psi \rangle_{L^2(\Omega)}$ for any $\psi \in H^1_0(\Omega)$. Using integration by parts, we have $\langle \nabla \cdot (\sigma^{m_l}\nabla u_i^{m_l}), \psi \rangle_{L^2(\Omega)} =-\langle \sigma^{m_l}\nabla u_i^{m_l}, \nabla\psi \rangle_{L^2(\Omega)}$. From the above discussion, the sequence of products $\sigma^{m_l}\nabla u_i^{m_l}$ is weakly convergent in $L^2(\Omega)$, that is, $\langle \sigma^{m_l}\nabla u_i^{m_l}, \nabla \psi \rangle_{L^2(\Omega)} \to 
\langle \sigma^* \nabla u_i^* , \nabla \psi \rangle_{L^2(\Omega)}$. With this preparation and using the continuity of the maps $u_i(\sigma)$, it follows that $(u_1^*,u_2^*)= (u_1(\sigma^*),u_2(\sigma^*))$, and the triplet $(\sigma^*,u_1^*,u_2^*)$ minimizes the objective $\hat{J}$. 
\end{proof}
In the case $\delta\neq 0$, we first note that the function $\log(1+z^2)$ is not convex. Therefore the PM functional, and, hence, the functional $\hat{J}$ in (\ref{functional}) is not weakly lower semi-continuous on $W^{1,p}(\Omega)$ for any $1 < p < \infty$. Nevertheless, $\hat{J}$ is a bounded below, lower semi-continuous Lipschitz functional, for which a minimizer exists, provided that $L_{ad}$ is compact.

\subsection{Characterization of local minima}
To characterize the solution of our optimization problem through first-order optimality conditions, we write the reduced functional $\hat{J}$ as 
$$
\hat{J} = \hat{J}_1+\hat{J}_2,~ J_i:L_{ad}\rightarrow\mathbb{R}^+,~ i=1,2,
$$
 where
\begin{equation}\label{sumoffunctionals}
\begin{aligned}
&\hat{J}_1(\sigma) = \dfrac{\alpha_1}{2}\|e^\sigma|\nabla u_1|-g_1^{\delta}\|^2_{L^{2}(\Omega)} + \dfrac{\alpha_2}{2}\|e^\sigma|\nabla u_2|-g_2^{\delta}\|^2_{L^{2}(\Omega)} +\dfrac{\beta}{2} \|\sigma\|^2_{L^2(\Omega)}\\
&+\dfrac{\delta}{2}\int_\Omega \ln(1+|\nabla \sigma(x,y)|^2)~dxdy,\\
&\hat{J}_2(\sigma) = \gamma \|\sigma\|_{L^{1}(\Omega)}.
\end{aligned}
\end{equation}

\begin{remark}
The functional $\hat{J}_1$ is smooth and possibly non-convex, while $\hat{J}_2$ is non-smooth and convex.

\end{remark}

We next state some properties of the reduced functional $\hat{J}_1(\sigma)$ which can be proved using the arguments in \cite[Lemma 3.1]{kuchment}.
\begin{proposition}\label{diff_J1}
The reduced functional $\hat{J}_1(\sigma)$ is weakly lower semi-continuous, bounded below and Fr\'{e}chet differentiable.
\end{proposition}
We now define the subdifferential of a non-smooth functional.

\begin{definition}[Subdifferential]
If $\hat{J}$ is finite at a point $\sigma$, the Fréchet subdifferential of $\hat{J}$ at $\sigma$ is defined as follows \cite{ekeland}

\begin{equation}\label{subdifferential}
\partial \hat{J}(\bar{\sigma}):=\Bigg\lbrace{\phi\in L^*_{ad}:\liminf_{\sigma\rightarrow \bar{\sigma}}\dfrac{\hat{J}(\sigma)-\hat{J}(\bar{\sigma})-\langle\phi,\sigma-\bar{\sigma}\rangle}{\|\bar{\sigma}-\sigma\|_2}}\geq 0\Bigg\rbrace,
\end{equation}
where $L^*_{ad}$ is the dual space of $L_{ad}$. An element $\phi\in \partial \hat{J}(\sigma)$ is called a subdifferential of $\hat{J}$ at $\sigma$.
\end{definition}

In our setting, we have the following
\[
\partial\hat{J}(\sigma) = \nabla \hat{J}_1(\sigma)+\partial \hat{J}_2(\sigma),
\]
since $\hat{J}_1$ is Fr\'{e}chet differentiable by Prop. \ref{diff_J1}. Moreover, for each $\alpha >0$, it holds that 
\[
\partial(\alpha \hat{J}) = \alpha \partial \hat{J}.
\]
The following proposition gives a necessary condition for a local minimum of $\hat{J}$ (see \cite{Roy_AET2}).
\begin{proposition}[Necessary condition]
If $\hat{J}=\hat{J}_1+\hat{J}_2$, with $\hat{J}_1, \hat{J}_2$ given by (\ref{sumoffunctionals}), attains a local minimum at $\sigma^*\in L_{ad}$, then
\[
0\in \partial \hat{J}(\sigma^*),
\]
or equivalently
\[
-\nabla \hat{J}_1(\sigma^*)\in \partial \hat{J}_2(\sigma^*).
\]
\end{proposition}
The following variational inequality holds for each $\lambda\in \partial \hat{J}_2(\sigma^*)$ (see \cite{stadler}).
\begin{equation}\label{var_ineq}
\langle\nabla \hat{J}_1(\sigma^*)+\lambda,\sigma-\sigma^*\rangle \geq 0,\qquad \forall \sigma \in L_{ad}.
\end{equation}
Using the definition of $\hat{J}_2$ in (\ref{sumoffunctionals}) and the fact that $L_{ad}$ is reflexive, the inclusion $\lambda\in\partial \hat{J}_2(\sigma^*)$ gives the following characterization of space of $\lambda$
\[
\lambda \in \Lambda_{ad}:=\lbrace \lambda \in L^2(\Omega): 0\leq\lambda\leq \gamma,\mbox{ a.e. in } \Omega\rbrace.
\]

A pointwise analysis of the variational inequality (\ref{var_ineq}) leads to the existence of a non-negative functions $\lambda_{\sigma_l}^*,\lambda_{\sigma_u}^*\in L^2(\Omega)$ that correspond to Lagrange multipliers for the inequality constraints in $L_{ad}$. We, thus, have the following first-order optimality system.

\begin{proposition}[First-order necessary conditions]\label{necessary}
The optimal solution of the minimization problem (\ref{reduced_min}) can be characterized by the existence of $(\lambda^*,\lambda_{\sigma_l}^*,\lambda_{\sigma_u}^*)\in \Lambda_{ad}\times L^2(\Omega)\times L^2(\Omega)$ such that 
\begin{eqnarray}
&&\label{grad1}\nabla_\sigma \hat{J}_1(\sigma^*) + \lambda^*+\lambda_{\sigma_u}^*-\lambda_{\sigma_l}^*=0,\\
&&\label{comp_st} \lambda_{\sigma_u}^* \geq 0,~ \sigma_u-\sigma^*\geq 0,~\langle \lambda_{\sigma_u}^*,\sigma_u-\sigma^* \rangle=0,\\ 
&&\lambda_{\sigma_l}^* \geq 0,~ \sigma^*-\sigma_l\geq 0,~\langle \lambda_{\sigma_l^*},\sigma^*-\sigma_l \rangle=0,\\ 
&&\label{4}\lambda^*=\gamma \mbox{ a.e. on } \lbrace x\in\Omega:\sigma^*(x) > 0 \rbrace,\\
&& \label{comp_end} 0\leq\lambda^*\leq \gamma \mbox{ a.e. on } \lbrace x\in\Omega:\sigma^*(x) = 0 \rbrace.
\end{eqnarray}

The conditions (\ref{comp_st})-(\ref{comp_end}) are known as the complementarity conditions for $(\sigma^*,\lambda^*)$.
\end{proposition}

To determine the gradient $\nabla_\sigma \hat{J}_1$, we use the adjoint approach (see for e.g., \cite{ RoyAnnunziatoBorzi2016,Roy2018}). This gives the following reduced gradient of $\hat{J}_1$ 
\begin{equation}\label{L2_gradients}
\begin{aligned}
\nabla_\sigma \hat{J}_1(\sigma^*)=&\alpha_1(e^{\sigma^*} |\nabla u_1|-g_1^\delta)|\nabla u_1|+\alpha_2(e^{\sigma^*} |\nabla u_2|-g_2^\delta)|\nabla u_2| + \nabla u_1 \cdot \nabla v_1 + \nabla u_2 \cdot \nabla v_2 + \beta \sigma^*\\
&-\delta\nabla\cdot \Bigg(\dfrac{\nabla \sigma^*}{1+|\nabla\sigma^*|^2}\Bigg)
\end{aligned}
\end{equation}
where $u_1,u_2$ satisfy the forward equations $\mathcal{L}(u_1,\sigma^*,f_D^1)=0, ~\mathcal{L}(u_2,\sigma^*,f_D^2)=0$, respectively, and $v_1,v_2$ satisfy
the adjoint equations

\begin{equation}\label{adj_1}
\begin{aligned}
-\nabla\cdot(e^{\sigma^*}\nabla{v_1})&=\alpha_1\nabla\cdot \Bigg [e^{\sigma^*}(e^{\sigma^*}|\nabla u_1|-g_1^{\delta})\sgn(\nabla u_1)\Bigg ]~ \mbox{in }\Omega, \\
v_1|_{\Gamma}&=0,\\
\end{aligned}
\end{equation}

\begin{equation}\label{adj_2}
\begin{aligned}
-\nabla\cdot(e^{\sigma^*}\nabla{v_2})&=\alpha_2\nabla\cdot \Bigg [e^{\sigma^*}(e^{\sigma^*}|\nabla u_2|-g_2^{\delta})\sgn(\nabla u_2)\Bigg ]~ \mbox{in }\Omega, \\
v_2|_{\Gamma}&=0.\\
\end{aligned}
\end{equation}

The complementarity conditions (\ref{comp_st})-(\ref{comp_end}) can be rewritten in a compact form as follows. Define 
\begin{equation}\label{mu}
~\mu^* = \lambda^*+\lambda_{\sigma_u}^*-\lambda_{\sigma_l}^*.
\end{equation}
Then the triplet $(\lambda^*,\lambda_{\sigma_l}^*,\lambda_{\sigma_u}^*)$ is obtained by solving the following equations
\begin{equation}\label{lambda_comp}
\begin{aligned}
&\lambda^* = \min(\gamma,\max(0,\mu^*)),\\
&\lambda_{\sigma_l}^* = -\min(0,\mu^*+\gamma),\\
&\lambda_{\sigma_u}^* = \max(0,\mu^*-\gamma),\\
\end{aligned}
\end{equation}
(see \cite{stadler}).
For each $k\in\R^+$, define the following quantity
$$
\begin{aligned}
E(\sigma^*,\mu^*) = \sigma^*&-\max\lbrace 0,\sigma^*+k(\mu^*-\gamma)\rbrace+\max\lbrace 0,\sigma^*-\sigma_u+k(\mu^*-\gamma)\rbrace\\
&-\min\lbrace 0,\sigma^*+k(\mu^*+\gamma)\rbrace+\min\lbrace 0,\sigma^*-\sigma_l+k(\mu^*+\gamma)\rbrace.\\
\end{aligned}
$$

The following lemma determines the complementarity conditions (\ref{comp_st})-(\ref{comp_end}) in terms of $E$ (see \cite[Lemma 2.2]{stadler}).

\begin{lemma}\label{Comp_conditions}
The complementarity conditions (\ref{comp_st})-(\ref{comp_end}) are equivalent to the following
\begin{equation}\label{comp_equality}
E(\sigma^*,\mu^*)=0,
\end{equation}
where $\mu$ is defined in (\ref{mu}). 
\end{lemma}

Using the gradients in (\ref{L2_gradients}) and Lemma \ref{Comp_conditions}, the optimality conditions (\ref{grad1})-(\ref{comp_end}) for the CDII-SR problem can be rewritten as follows 
\begin{proposition}\label{nec_opt}
A local minimizer $(u_1, u_2,\sigma^*)$ of the problem (\ref{min_problem}) can be characterized by the existence of $(v_1,v_2,\mu^*)\in H^1_0(\Omega)\times H^1_0(\Omega)\times L_{ad}$, such that the following system is satisfied
\begin{equation}\label{f1}
\begin{aligned}
-\nabla\cdot(e^{\sigma^*}\nabla{u_1})&=0~ \mbox{in }\Omega, \\
u_1|_{\Gamma}&=f^1_D,\\
-\nabla\cdot(e^{\sigma^*}\nabla{v_1})&=\alpha_1\nabla\cdot \Bigg [e^{\sigma^*}(e^{\sigma^*}|\nabla u_1|-g_1^{\delta})\sgn(\nabla u_1)\Bigg ]~ \mbox{in }\Omega, \\
v_1|_{\Gamma}&=0,\\
-\nabla\cdot(e^{\sigma^*}\nabla{u_2})&=0~ \mbox{in }\Omega, \\
u_2|_{\Gamma}&=f^2_D,\\
-\nabla\cdot(e^{\sigma^*}\nabla{v_2})&=\alpha_2\nabla\cdot \Bigg [e^{\sigma^*}(e^{\sigma^*}|\nabla u_2|-g_2^{\delta})\sgn(\nabla u_2)\Bigg ]~ \mbox{in }\Omega, \\
v_2|_{\Gamma}&=0,\\
\alpha_1(e^{\sigma^*} |\nabla u_1|-g_1^\delta)e^{\sigma^*}|\nabla u_1|+\alpha_2(e^{\sigma^*} |\nabla u_2|-g_2^\delta )e^{\sigma^*}|\nabla u_2| &+\nabla u_1 \cdot \nabla v_1 + \nabla u_2 \cdot \nabla v_2 + \beta \sigma^*\\
&-\delta\nabla\cdot \Bigg(\dfrac{\nabla \sigma^*}{1+|\nabla\sigma^*|^2}\Bigg)+ \mu^*=0,~ \mbox{a.e. in }\Omega,\\
E(\sigma^*,\mu^*)&=0 , \qquad \mbox{a.e. in }\Omega.\\
\end{aligned}
\end{equation}

\end{proposition}

\section{Numerical solution of the CDII-SR problem}
\label{sec:numsol}

In this section, we discuss numerical optimization and approximation schemes to solve the CDII-SR problem. We also describe the Picard-type iterative scheme used by the authors in \cite{Hoff2014} to compare the results of our scheme with the Picard scheme.

\subsection{Variable inertial proximal method for CDII-SR} \label{sec:VIP}

 In this context, we first discuss proximal methods that consists of identifying a smooth and a non-smooth part in the reduced objective $\hat J(\sigma)$. Thus, we consider the following optimization problem 
\begin{align}\label{genprob}
\min\limits_{\sigma \in L_{ad}}\, \hat J(\sigma) := \hat{J}_1(\sigma)+\hat{J}_2(\sigma). 
\end{align}

We assume that $\nabla_\sigma \hat{J}_1(\sigma)$, given in (\ref{L2_gradients}) is Lipschitz continuous and the upper bound for the Lipschitz constant is obtained using a backtracking search scheme, which will be discussed later. Also, from (\ref{sumoffunctionals}), we have that $\hat{J}_2(\sigma)$ is a continuous, convex, and nondifferentiable functional. The formulation of proximal methods depends, essentially, on the following lemma \cite{Roy_AET2}

\begin{lemma}\label{lem:ineq}
Let $\hat{J}_1$ be differentiable with a Lipschitz continuous gradient with Lipschitz constant $L(\hat{J}_1)$. Then the following holds 
\begin{align}\label{mainlem}
\hat{J}_1(\sigma)\leq \hat{J}_1(\widetilde\sigma)+\left< {\nabla}\hat{J}_1(\widetilde\sigma),\sigma-\widetilde\sigma\right>+\frac {L} 2\|\sigma-\widetilde\sigma\|^2,\quad\forall \sigma,\widetilde\sigma\in L_{ad} ,
\end{align}
for all $L\geq L(\hat{J}_1)>0$. 
\end{lemma}

We note that $L:=L(\hat{J}_1)$ represents the smallest value of $L$ such that (\ref{mainlem}) holds true. 

In a proximal scheme, one usually minimizes an upper bound of the objective functional at each iteration, instead of minimizing the functional directly. From Lemma \ref{lem:ineq}, we obtain the following
$$
\min_{\sigma\in L_{ad}}\left\{\hat{J}_1(\sigma)+\hat{J}_2(\sigma)\right\}\leq\min_{\sigma\in L_{ad}}\left\{\hat{J}_1(\widetilde\sigma)+\left< {\nabla}\hat{J}_1(y),\sigma-\widetilde\sigma\right>+\frac {L} 2\|\sigma-\widetilde\sigma\|^2+\hat{J}_2(\sigma)\right\},
$$
where equality holds if $\sigma=\widetilde\sigma$. Furthermore, we have the following equation 
\begin{align}
&\argmin_{\sigma\in L_{ad}}\left\{\hat{J}_1(\widetilde\sigma)+\left< {\nabla}\hat{J}_1(\widetilde\sigma),\sigma-\widetilde\sigma\right>+\frac {L} 2\|\sigma-\widetilde\sigma\|^2 + \hat{J}_2(\sigma)\right\} \nonumber \\
&=\argmin_{\sigma\in L_{ad}}\left\{\frac {L} 2\left\|\sigma-\left(\widetilde\sigma-\frac1{L} {\nabla}\hat{J}_1(\widetilde\sigma)\right)\right\|^2 + \hat{J}_2(\sigma)\right\}. \label{proxeq}
\end{align}

Using the definition of $\hat{J}_2(\sigma)=\gamma\|\sigma\|_{L^1(\Omega)}$, we have the following lemma from \cite{schindeleEll} that helps in characterizing the solution of (\ref{proxeq}).

\begin{lemma} 
The following equation holds
\begin{align*}
\argmin_{\sigma\in L_{ad}}\left\{\tau\|\sigma\|_{L^1}+\frac 1 2\|\sigma-\widetilde\sigma\|^2\right\}=\mathbb{S}^{L_{ad}}_{\tau}(\widetilde\sigma)\quad \text{ for any }\;\widetilde\sigma\in L^2(\Omega),
\end{align*}
where the left-hand side represents the proximal function and the projected soft thresholding function on the right-hand side is defined as follows 
\begin{align}
\mathbb{S}^{L_{ad}}_{\tau}(\widetilde\sigma):=\begin{cases}
\min\{\widetilde\sigma-\tau,\sigma_u\}&\text{on }\{(x,y)\in\Omega: \, \widetilde\sigma(x,y)>\tau\}\\
0&\text{on }\{(x,y)\in\Omega: \, |\widetilde\sigma(x,y)|\leq\tau\}  \label{Sfunction}\\
\max\{\widetilde\sigma+\tau,\sigma_l\}&\text{on }\{(x,y) \in\Omega: \, \widetilde\sigma(x,y)<-\tau\}
\end{cases}.
\end{align}
\end{lemma}

Using this lemma, the solution to (\ref{proxeq}) is given by
\begin{align*}
\argmin_{\sigma \in L_{ad}}\left\{\hat{J}_2(\sigma)+\frac {L} 2\left\|\sigma -\left(\widetilde\sigma-\frac1{L} {\nabla}\hat{J}_1 (\widetilde\sigma)\right)\right\|^2\right\}=\mathbb{S}^{L_{ad}}_{\frac{\gamma}{L}}\left(\widetilde\sigma-\frac1{L} {\nabla}\hat{J}_1(\widetilde\sigma)\right) .
\end{align*} 
This gives rise to the following iterative scheme 
\begin{align*}
\sigma_{k+1} \leftarrow \mathbb{S}^{L_{ad}}_{\frac{\gamma}{L}}\left(\sigma_{k}-\frac1{L} {\nabla}\hat{J}_1(\sigma_{k})\right) ,
\end{align*}
starting from a given $\sigma_0$ and is known as the iterative shrinkage-thresholding algorithm (ISTA) scheme \cite{schindeleEll}. We note that the argument of $\mathbb{S}^{L_{ad}}_{\frac{\gamma}{L}}$ represents a gradient update in a steepest descent scheme with a fixed step size $s= 1/L$ in conjunction with a regularized PM filter \cite{Roy_AET2}. Further, to accelerate the ISTA scheme described above, one can consider a sequence $\{t_k,v_k\}$ \cite{schindeleEll, schindele} such that
\begin{equation}\label{algot}
t_0=1,\qquad t_k:=1+\sqrt{1+4t_{k-1}^2}/2 ,
\end{equation}
and
\begin{align}\label{vupdate}
v_0:=\sigma_0,\qquad v_k:=\sigma_{k}+\frac{(t_{k-1}-1)}{t_{k}}(\sigma_{k}-\sigma_{k-1}). 
\end{align}
This gives us the following update for the optimization variable $\sigma_k$
\begin{equation}\label{FISTA}
\sigma_{k+1}\leftarrow \mathbb{S}^{L_{ad}}_{\frac{\gamma}{L}}\left(v_{k}-\frac1{L} {\nabla}\hat{J}_1(v_{k})\right).
\end{equation}

Replacing $v_k$ in \eqref{FISTA} with \eqref{vupdate}, and assuming that ${\nabla}\hat{J}_1(\sigma_{k}) \approx {\nabla}\hat{J}_1(v_{k})$, we obtain the following iterative scheme \cite{schindele}
\begin{align}\label{update}
\sigma_{k+1}\leftarrow \mathbb{S}^{L_{ad}}_{\gamma \, s_k}\left(\sigma_{k}-s_k\nabla_\sigma \hat{J}_1(\sigma_{k})+\theta_k \, (\sigma_k-\sigma_{k-1})\right) ,
\end{align}
where $\sigma_{-1}=\sigma_0$.  

The above discussion is valid for any $L\geq L(\hat{J}_1)$. However, since the quantity $s=1/L$ represents the step size in a gradient update, we use a backtracking line search algorithm to determine the optimal step size in each iteration. This leads to the computation of an upper bound $L_k$ that satisfies $L_k\geq L(\hat{J}_1)$ at each iteration step. Thus, we define our variable step size as $s_k = 1/L_k$ and substitute $\tau$ in (\ref{Sfunction}) with $\gamma s_k$. The variable step size causes the factor $\frac{(t_{k-1}-1)}{t_{k}}$ in (\ref{vupdate}) to be non-optimal and we replace it by the fixed inertial parameter $\theta$. This leads to a variable inertial proximal (VIP) scheme, which is described in Algorithm \ref{proximal_algo}. 

With our VIP scheme, we aim at determining an optimal $\sigma \in L_{ad} \subset H^1_0(\Omega)$. But in the update step of the algorithm, we have the argument of the thresholding function $\mathbb{S}^{L_{ad}}_{\frac{\gamma}{L}}$ as $\sigma_{k}-s_k\nabla_\sigma \hat{J}_1(\sigma_{k})$. The term $\nabla_\sigma \hat{J}_1(\sigma_{k})$ is only in $L^2(\Omega)$ and the resulting update gives us the argument of $\mathbb{S}^{L_{ad}}_{\frac{\gamma}{L}}$ in $L^2(\Omega)$, which is not desired. We, thus, use the $H^1$ gradient instead of the $L^2$ gradient, which are related by the equation $((\nabla_\sigma \hat{J}_1)_{H^1},v)_{H^1(\Omega)}=(\nabla_\sigma \hat{J}_1,v)_{L^2(\Omega)}$ for all $v \in H^1(\Omega)$. 
But such a $H^1$ gradient results in a highly diffused $\sigma$ with blurred edges. We, instead, consider a weighted $H^1$ product 
that represents a suitable denoising of the $\nabla_\sigma \hat{J}_1(\sigma)$. We apply the denoising operator $R(c)=(I-c \, \Delta)^{-1}$ with a small denoising parameter $c$ and define $(\nabla_\sigma \hat{J}_1)_{H^1} = R(c) \, \nabla_\sigma \hat{J}_1$. Note that a higher value of $c$ results in a greater blurring of the edges along with noise removal. On the other hand, since the PM term in the functional $J$ promotes better resolution with edge-enhancement, we choose the value of $c$ in proportion to the weight $\delta$ of the PM functional term.

We summarize the variable inertial proximal (VIP) scheme for our CDII-SR setup in Algorithm \ref{proximal_algo} below

\begin{algorithm}[Variable inertial proximal (VIP) method]\label{proximal_algo}\
\begin{enumerate}

\item  Input: $\beta$, $\hat{J}_1$, $\sigma_0=\sigma_{-1}$, $L_{ad}$, $TOL$, $n>1$, $L_0>0$\\
   \textbf{Initialize:} $E_0=1$, $k=0$, choose $\theta\in(0,1)$ and $c_1 < 2$ and $c_2>0$;
\item While{$\|E_{k-1}\|>TOL$} do
\item Compute $\nabla_\sigma \hat{J}_1(\sigma_{k})$ 
\item Backtracking: Find the smallest nonnegative integer $i$ such that with \\
\hspace*{3ex}$\tilde{L}=n^{i}L_{k-1}$
\begin{align*}
\hat{J}_1(\tilde{\sigma})\leq\hat{J}_1(\sigma_{k})+\left< \nabla_\sigma \hat{J}_1(\sigma_{k}),\tilde{\sigma}-\sigma_{k}\right>+\frac {\tilde{L}} 2\|\tilde{\sigma}-\sigma_{k}\|^2
\end{align*}
\hspace*{3ex}where $\tilde{\sigma}=\mathbb{S}^{L_{ad}}_{\gamma\, s}\left(\sigma_{k}-s\, (\nabla_\sigma \hat{J}_1)_{H^1}(\sigma_{k})+\theta(\sigma_k-\sigma_{k-1})\right)$, $s = c_1 (1-\theta)/(\tilde{L}+2c_2)$,
\item  Set $L_k=\tilde{L}$ and $s_k=c_1 (1-\theta)/(L_k+2c_2)$.
\item  $\sigma_{k+1}=\mathbb{S}^{L_{ad}}_{\gamma\, s_k}\left(\sigma_{k}-s_k \, (\nabla_\sigma \hat{J}_1)_{H^1}(\sigma_{k})+\theta(\sigma_k-\sigma_{k-1})\right)$
 \item  $\mu_k=-\alpha \sigma_k-(\nabla_\sigma \hat{J}_1)_{H^1}(\sigma_{k})$ 
 \item  $E_{k}=E(\sigma_k,\mu_k)$
 \item $k=k+1$   
\item end
\end{enumerate}
\end{algorithm}
In the VIP algorithm, we need to compute the reduced gradient $\nabla_\sigma \hat{J}_1$. This, in turn, requires an accurate numerical solution of the forward and the corresponding adjoint EIT problems  as given in Proposition \ref{nec_opt}. For the forward EIT equation (\ref{eq1}), we use the cell-nodal finite-difference approximation. We consider a sequence of uniform grids 
$\lbrace\Omega_h\rbrace_{h>0}$ given by 
$$
\Omega_h = \lbrace{(x_i,y_j)\in\mathbb{R}^2:(x_i,y_j) = (a+ih,a+jh),~(i,j)\in
\lbrace{0,\hdots,N}\rbrace^2}\rbrace\cap\Omega,
$$
where $N$ represents the number of cells in each direction and $h = \dfrac{(b-a)}{N}$ is the mesh size. The corresponding cell-nodal scheme for (\ref{eq1}), at the grid point $(x_i,y_j)$,  is given as follows 

\begin{equation}\label{FD_scheme}
\begin{aligned}
&\dfrac{1}{h^2} \Bigg\lbrace({e^{\sigma_{i+1/2,j}}}+{e^{\sigma_{i-1/2,j}}}+e^{\sigma_{i,j+1/2}}+e^{\sigma_{i,j-1/2}})u_{i,j}\\
& -e^{\sigma_{i+1/2,j}}u_{i+1,j}- e^{\sigma_{i-1/2,j}}u_{i-1,j}- e^{\sigma_{i,j+1/2}}u_{i,j+1}- e^{\sigma_{i,j-1/2}}u_{i,j-1}  \Bigg\rbrace=0,\qquad 1 \leq i,j \leq N-1,\\
\end{aligned}
\end{equation}
where $\sigma_{i \pm 1,j}=\sigma(x_i \pm h , y_j)$, 
$\sigma_{i ,j \pm 1}=\sigma(x_i , y_j \pm h )$. The required intermediate values 
of $\sigma$ are computed as follows  
$$
{\sigma_{i \pm 1/2,j}} = \dfrac{1}{2}\Bigg({\sigma_{i\pm 1,j}}+{\sigma_{i,j}} \Bigg) \quad \mbox{ and } \quad 
{\sigma_{i ,j \pm 1/2}} = \dfrac{1}{2}\Bigg({\sigma_{i,j \pm 1}}+{\sigma_{i,j}} \Bigg) .
$$
The Dirichlet boundary data $f_D$ is included in the usual way in the right-hand side of the algebraic equation. 

For the adjoint equations (\ref{adj_1}) and (\ref{adj_2}), we first note that the cell nodal finite difference scheme is not applicable to the right-hand side term in both the equations as they are of the form $G=\nabla\cdot F$, where $F$ is in $L^2(\Omega)$. We modify the cell nodal scheme by replacing the nodal value of $G$ at $(x_i,y_j)$ with a cell average of $G$ given as follows
\[
G_a = \frac{1}{h^2}\int_{C_{ij}}G(x,y)~dxdy,
\] 
where the cell $C_{ij}$ is defined by
\[
C_{ij}:= \left(x_i-\frac{h}{2},x_i+\frac{h}{2}\right)\times \left(y_j-\frac{h}{2},y_j+\frac{h}{2}\right),\quad 1\leq i,j\leq N-1.
\]
Since $G=\nabla\cdot F$, using the divergence theorem we have
\[
G_a = \frac{1}{h^2}\int_{C_{ij}}\nabla\cdot F(x,y)~dxdy =\frac{1}{h^2}\int_{\partial C_{ij}} F(x,y).n~ds 
\]
The above integral can be approximated with a midpoint quadrature rule along each edge of $C_{ij}$. This results in the following approximation
\[
G_a = \frac{F^1_{i+1/2,j}-F^1_{i-1/2,j}}{h}+\frac{F^2_{i,j+1/2}-F^2_{i,j-1/2}}{h},
\]
where $F=(F^1,F^2)$.

\subsection{Picard-type algorithm}\label{sec:Picard}
To solve the CDII inverse problem, the authors in \cite{Hoff2014} use a Picard-type iterative algorithm. This algorithm was the fundamental approach that was used to solve the CDII inverse problem expressed as the 1-Laplacian equation given in (\ref{1_Laplacian}) \cite{Na09,Na11} was known as the J-substitution algorithm. To start the Picard algorithm, we use an initial guess for $\sigma$ as $\sigma_0$. We then solve \eqref{eq1} for $u_1$ in the first iteration and then obtain $\sigma_1$ from the interior data \eqref{interior_data} corresponding to this value of $u_1$. In the next iteration, we solve for $u_2$ using $\sigma_1$ and again obtain the next iterate $\sigma_2$ from the interior data \eqref{interior_data} corresponding to this value of $u_2$. We apply this method repeatedly, going back and forth between the two sets of interior data, till the difference in the $L^2$ norm of two successive iterates of $\sigma$ are small. The algorithm is summarized below.

\begin{algorithm}[Picard-type algorithm]\label{Picard_algo}\
\begin{enumerate}

\item  Input: Initial guess $\sigma_0$, maximum iterations = $K$, $TOL$\\
   \textbf{Initialize:} $err_0=1$, $k=0$.
\item While{ $\|err_{k}\|_2>TOL$ and $k<K$} do
\item $j = (k ~MOD ~2) + 1$ (To go back and forth between the two interior measurements.
\item Solve for $u_j^k$ from (\ref{eq1}), i.e., solve
\[
\begin{aligned}
\mathcal{L}(u_j^k,\sigma_k,f_D^j) = 0
\end{aligned}
\]
\item  Update $\sigma_{k+1}  = \dfrac{H_j}{|\nabla u_j^k|}$
\item  $err_{k+1} = \|\sigma_{k+1}-\sigma_k\|_2$
 \item $k=k+1$   
\item end
\end{enumerate}
\end{algorithm}

The authors demonstrated in their paper that the Picard algorithm was fast and efficient. With 2 data sets and even in the presence of noisy interior data, the Picard method gave good conductivity reconstructions in CDII with a high resolution and less visible artifacts in comparison to other Newton based schemes. Thus, in order to test the performance of our CDII-SR scheme, we will compare the results of our scheme with the results of the Picard scheme in the next section.

\section{Numerical experiments}
\label{sec:numexp}
In this section, we validate our CDII-SR framework using different experiments that support the choice of our proposed  mathematical formulation and demonstrate its effectiveness in reconstructing a wide variety of objects. We choose the two boundary conditions as $f^1_D=x,~f_D^2=y$ on $\Gamma$, which is the boundary of $\Omega=(-1,1) \times (-1,1)$.  For the experiments, we test our scheme with different values of the regularization parameters $\beta,\gamma,\delta$ in the functional \eqref{functional}, relative to the  value of the data-fit parameters $\alpha_1,\alpha_2 $. For this purpose, we fix the values of $\alpha_1 = \alpha_2 = 1.0$. Since the primary objective is to fit the measured interior data, the value of the regularization parameters are chosen to be less than 1. The parameters of our VIP scheme are chosen based on the convergence properties of the scheme given in \cite{schindele}. We specified the tolerance level as $TOL=10^{-4}$ and maximum number of iterations as 20. But, due to the high non-linearity of the problem, our VIP scheme terminates because of the criteria of  maximum number of iterations. For the Picard algorithm, described in Section \ref{sec:Picard}, we choose the same stopping criterion. To solve the equation $\mathcal{L}(u_j^k,\sigma_k,f_D^j) = 0$ in Algorithm \ref{Picard_algo}, we use the finite difference scheme described in (\ref{FD_scheme}).

In all the experiments, the domain $\Omega$ is uniformly discretized into $N=150$ subintervals in both the $x$ and $y$ directions with $h=0.013$. The generation of the synthetic interior electric field data $H^\delta$ is done as follows: we first solve for $u$ in (\ref{eq1}) with given value of $\sigma$ on a finer mesh with $N=400$ using the finite difference method outlined in Section \ref{sec:numsol}. Then, we compute $\nabla u$ with one-sided finite differences to obtain $H^\delta$ on the finer mesh. In the final step, we restrict the obtained $H^\delta$ onto the coarser mesh with $N=150$ and choose this as our given data to which we also add noise in some of the experiments.

In Test Case 1, we consider a phantom that is represented by a disk centered at $(0.25,0.25)$ with radius 0.25. The value of $\sigma$ inside the disk is 1 with the background value chosen as 0. The plots of the actual $\sigma$ and the reconstructed $\sigma$ are shown in Figure \ref{disk}.

\begin{figure}[H]
\centering
\subfloat[Actual phantom]{\includegraphics[width=0.35\textwidth]{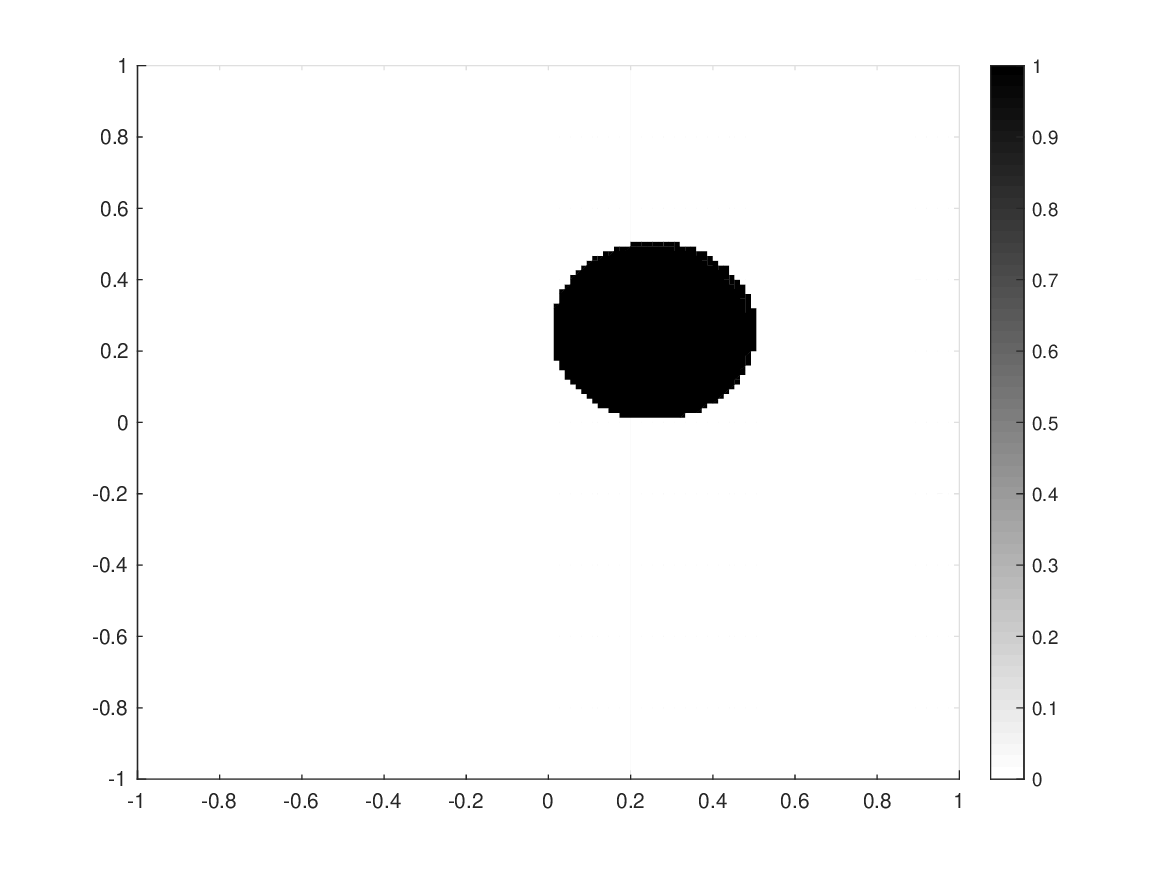}\label{disk_actual}}
\subfloat[$\beta=\gamma=\delta=0$]{\includegraphics[width=0.35\textwidth]{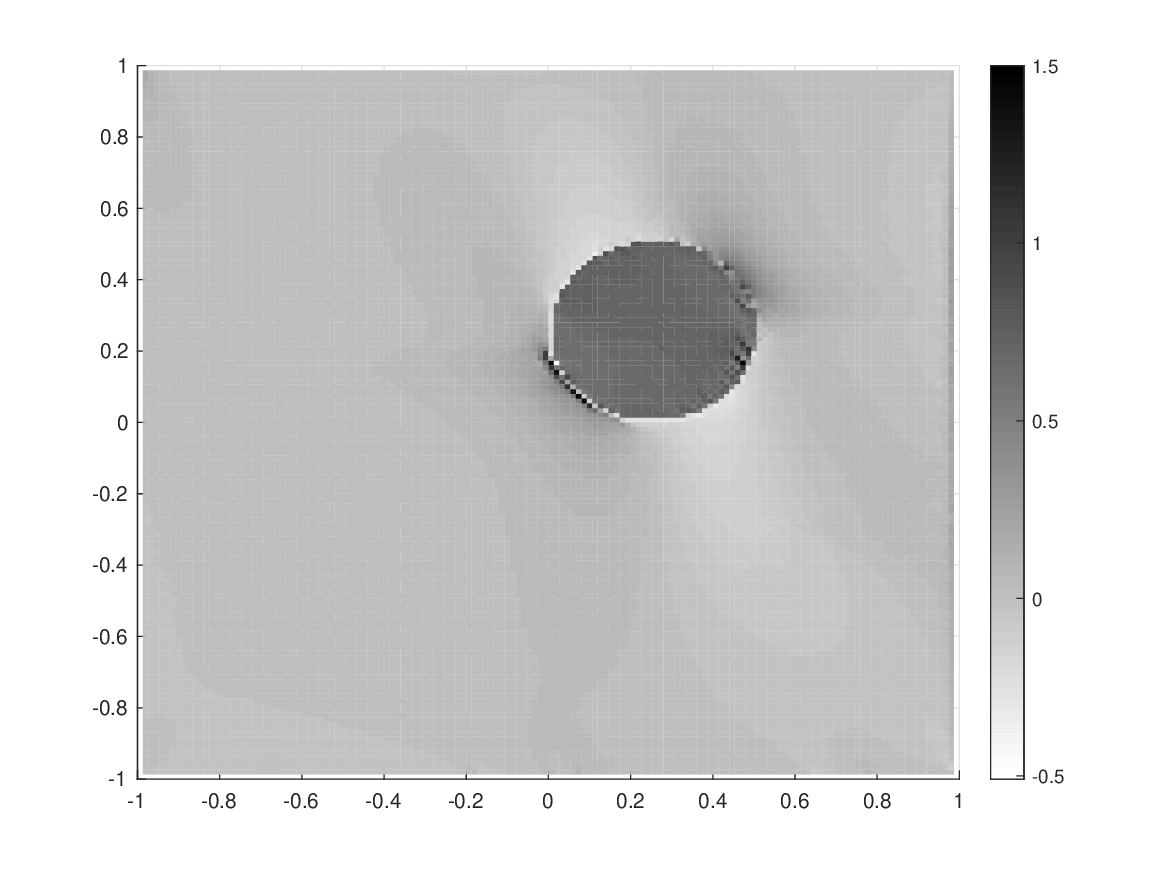}\label{disk_no_regularization}}
\subfloat[$\beta = 0.03,\gamma = 0.3,\delta = 0, c = 0$]{\includegraphics[width=0.35\textwidth]{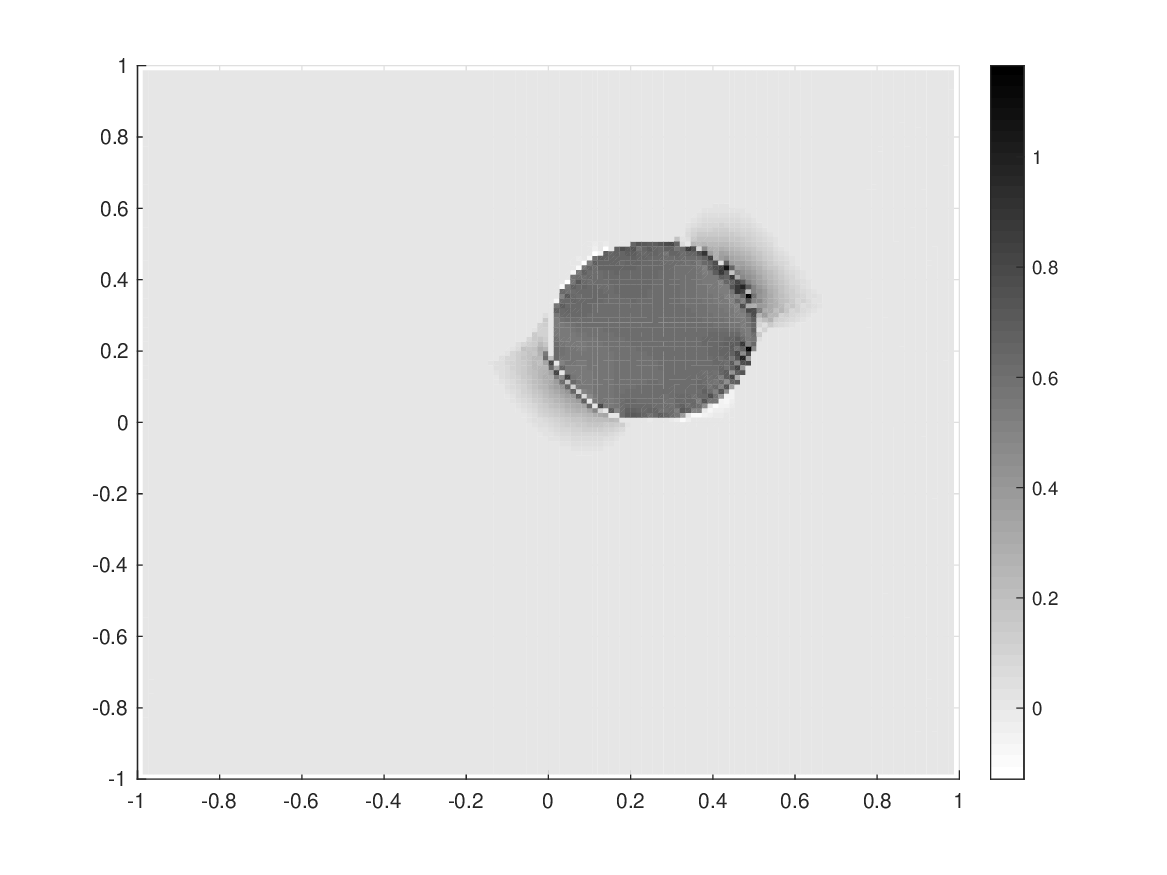}\label{disk_L2L1_regularization}}\\
\subfloat[$\beta = 0.03,\gamma = 0.3,\delta = 0,$\newline$ c = 0.001$]{\includegraphics[width=0.35\textwidth]{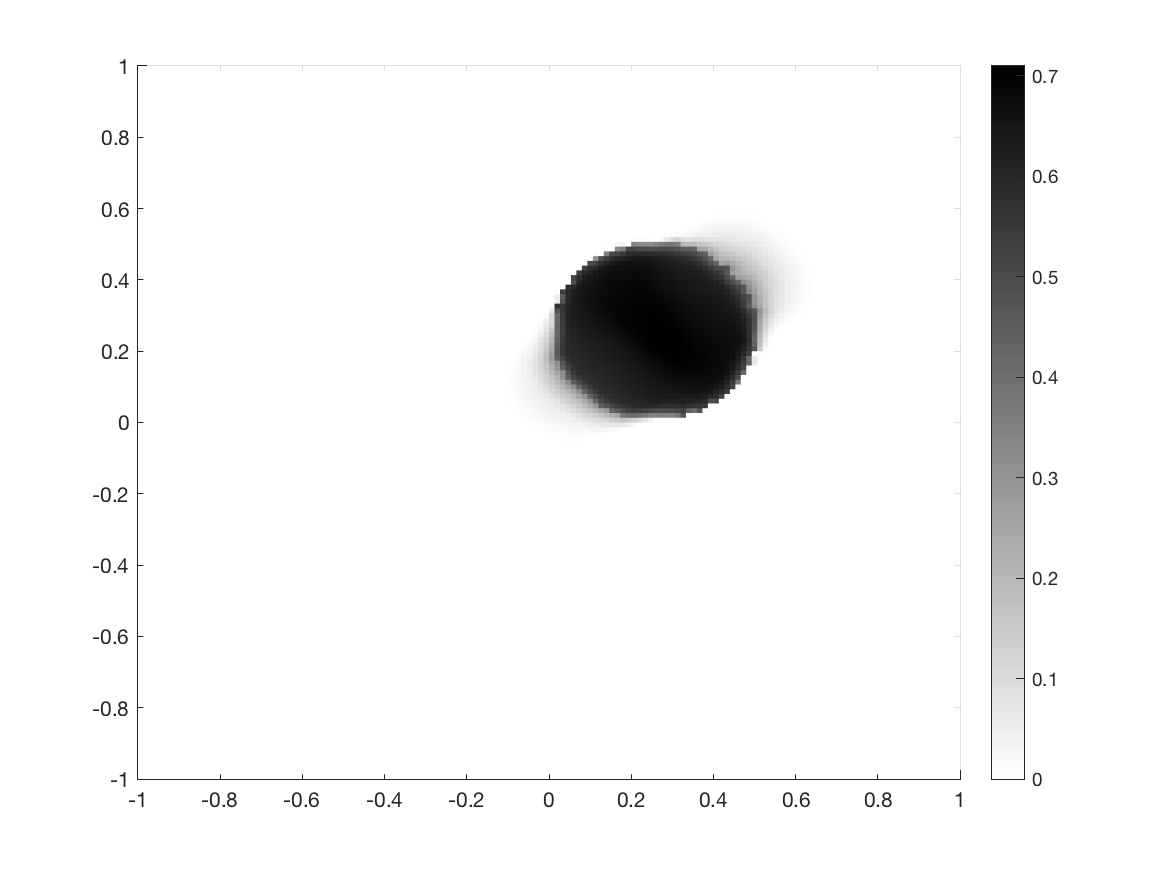}\label{disk_L2L1denoising_regularization}}
\subfloat[$\beta = 0.03,\gamma = 0.3,\delta = 0.01, $ \newline$c = 0.001$]{\includegraphics[width=0.35\textwidth]{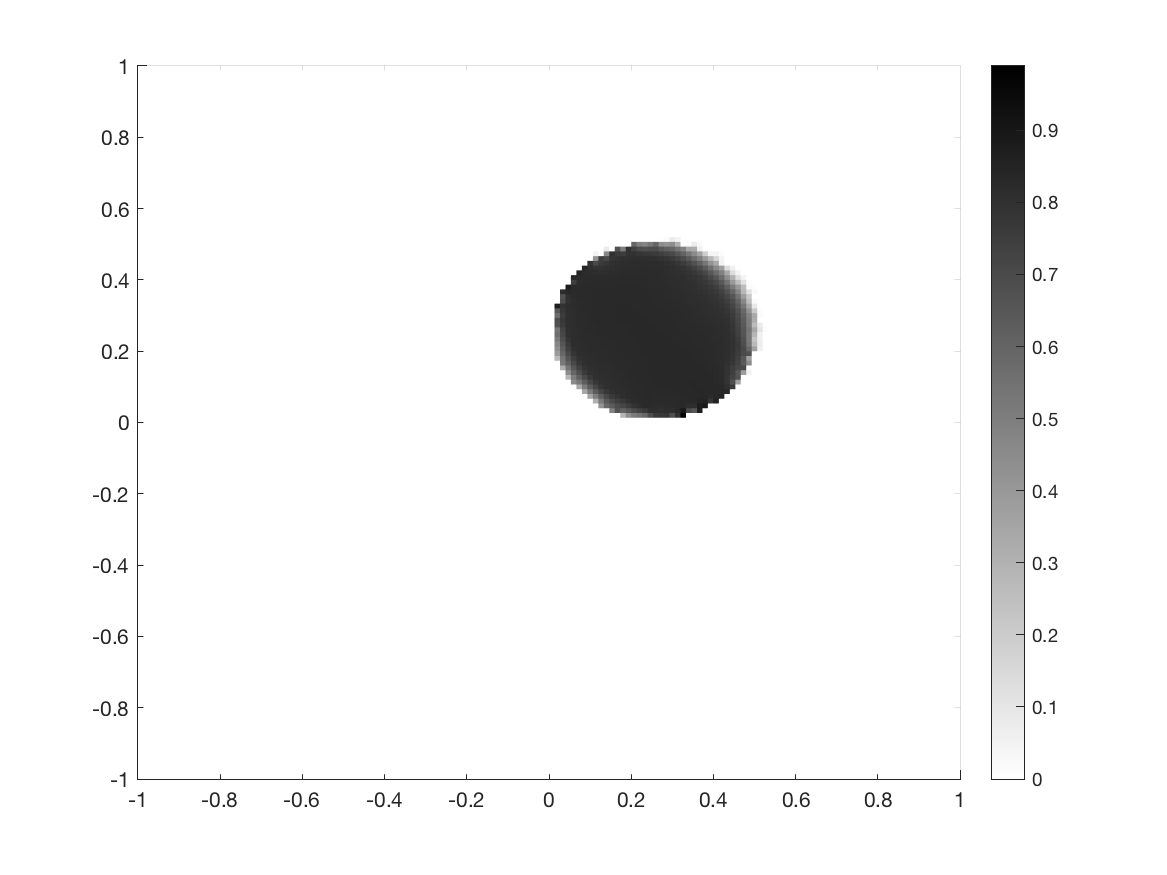}\label{disk_recon}}
\subfloat[ Picard algorithm in \cite{Hoff2014}]{\includegraphics[width=0.35\textwidth]{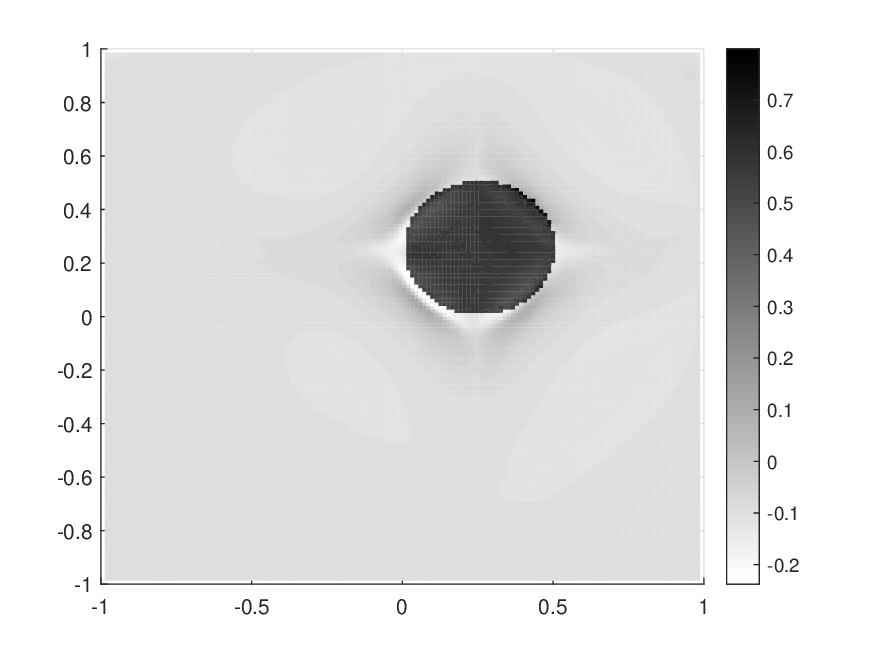}\label{disk_recon_lin}}

    \caption{Test Case 1- The actual and reconstructed disk 
    with different choices of the values of the regularization weights. }
    \label{disk}
  \end{figure}
  Figure \ref{disk_no_regularization} shows the reconstruction of $\sigma \in L_{ad}$ without any regularization terms, i.e. $\beta=\gamma=\delta=0$ and no denoising, i.e., $c=0$. Presence of strong artifacts can be observed in this case, which is inherent to the inverse problem and not the algorithm. A study of the pattern of such artifacts are very challenging and is out of the scope of the paper.

Figure \ref{disk_L2L1_regularization} shows the result using the CDII-SR scheme without the denoising and the Perona-Malik regularization term,  $c=\delta=0$, but with the $L^2-L^1$ regularization. The $L^2-L^1$ regularization parameter values were set at $\beta = 0.03,~\gamma = 0.3$. We observe that the artifacts are reduced to some extent, but are still present. Figure \ref{disk_L2L1denoising_regularization} shows the reconstruction with the same $L^2-L^1$ regularization parameter values and the $H^1$ denoising parameter value $c = 0.001$ but without the PM filter. In this case, we observe that the artifacts diminish by a huge amount, but the edges are more blunt and the value of $\sigma$ is lowered, leading to a loss of resolution and contrast. We correct this loss using the PM regularization term with the parameter value $\delta =0.01$. This can be observed in Figure \ref{disk_recon}, where the edges are fairly well seen and the recovered parameter values are very close to the true ones. We also compare our results with the reconstruction in Figure \ref{disk_recon_lin} obtained with the Picard algorithm proposed in \cite{Hoff2014}. We observe a lot of artifacts and a significant loss of contrast in Figure \ref{disk_recon_lin} in comparison to the reconstruction shown in Figure \ref{disk_recon}, which suggests that our CDII-SR scheme outperforms the Picard scheme.

We also implemented Test Case 1 with different values of the parameters $\theta,c_1,c_2$ in the VIP algorithm along with values of $\beta,\gamma,\delta = 0, 0.001, 0.01, 0.1$. The convergence results of the VIP algorithm are listed in Table \ref{tab:VIP}. The variable $v$ represents the value of each one of the parameters $\theta,c_1,c_2$. In the table, each cell in the columns under the individual parameters represent the convergence property of the VIP algorithm for $v$ in the prescribed range, when the other parameters are specified in their range of convergence.

\begin{table}[H]

\begin{center}
\begin{tabular}{V{3}cV{3}c|c|cV{3}}
\hlineB{3}
\textbf{Values} &$\theta$ & $c_1$ & $c_2$ \\
\hlineB{3}
$v \geq 2$  & No convergence & No convergence & Slow Convergence \\
\hline
$v < 2, v \approx 2$  & No convergence & Good convergence & Slow Convergence \\
\hline
$1 \leq v < 2$  & No convergence & Convergence & Slow Convergence \\
\hline
$v <1, v \approx 1$  &  Convergence & Convergence & Slow Convergence \\
\hline
$v \ll 1$  &  Convergence & Slow Convergence & Good Convergence \\
\hlineB{3}
\end{tabular}
\end{center}
\caption{Convergence properties of the VIP algorithm for various values $v$ of $\theta,c_1,c_2$}
\label{tab:VIP}

\end{table}

We first remark that varying the parameters in the VIP algorithm results in change of convergence properties of the CDII-SR scheme. But choosing different values of $\theta,c_1,c_2$ in the range of convergence, we did not observe any change in the qualitative or quantitative properties of the reconstructed conductivities. From Table \ref{tab:VIP}, we can see that for values of $\theta \in [0,1),~ c_1 \approx 2,~ c_2 \ll 1$ we have good convergence of the VIP scheme. For a more detailed discussion of these range of values for convergence, we refer the reader to \cite{schindele}. Thus, we choose $\theta=0.5,~ c_1 = 1.9, c_2 = 0.001$ in our subsequent experiments.  

In Test Case 2, we consider the heart and lung phantom. It consists of two ellipses representing lungs with the value of $\sigma= 1$ and a circular region representing the heart with the value of $\sigma=0.5$. The background value of $\sigma$ is 0.  The plots of the actual and reconstructed $\sigma$ for various values of $\beta,\gamma,\delta,c$ are shown in Figure \ref{heart_lung}.

\begin{figure}[H]
\centering
\subfloat[Actual phantom]{\includegraphics[width=0.35\textwidth]{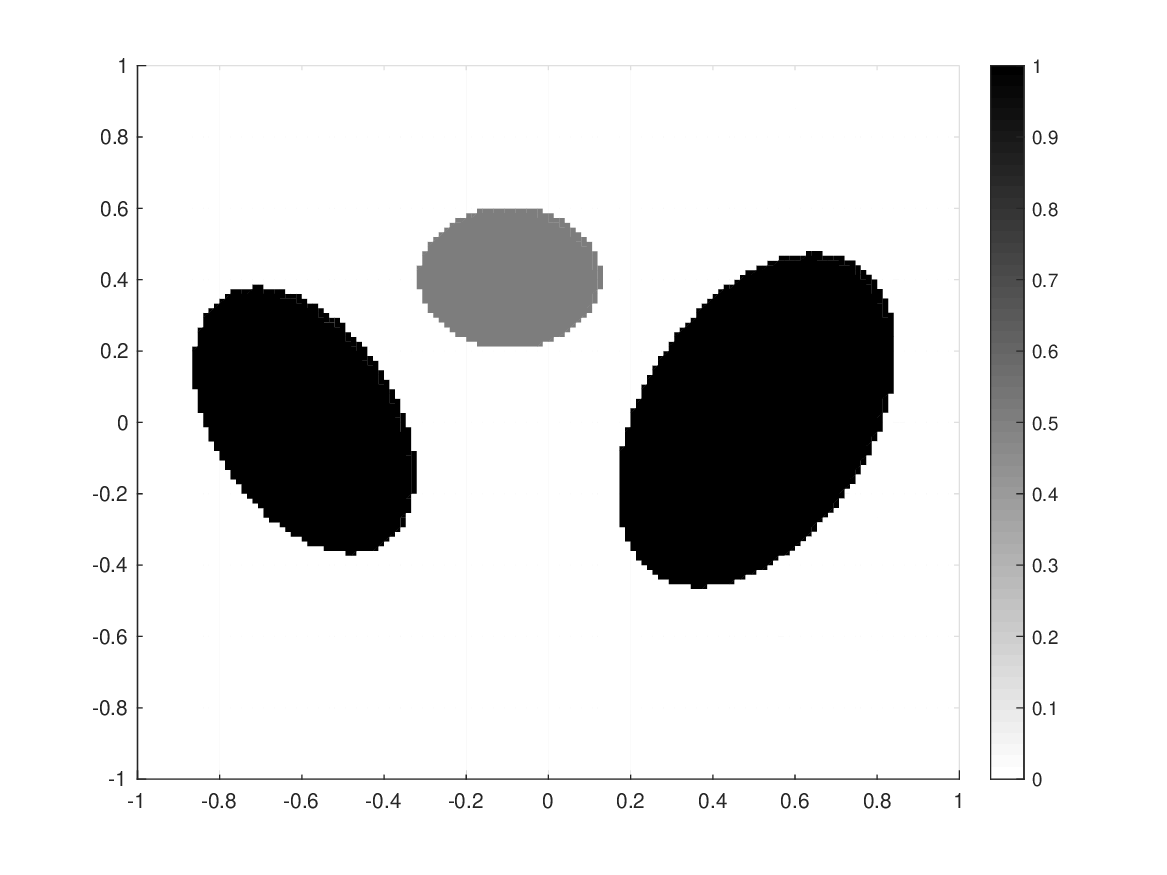}\label{heart_lung_actual}}
\subfloat[ $\beta=0.3,\gamma=0.01,\delta=0.01$]{\includegraphics[width=0.35\textwidth]{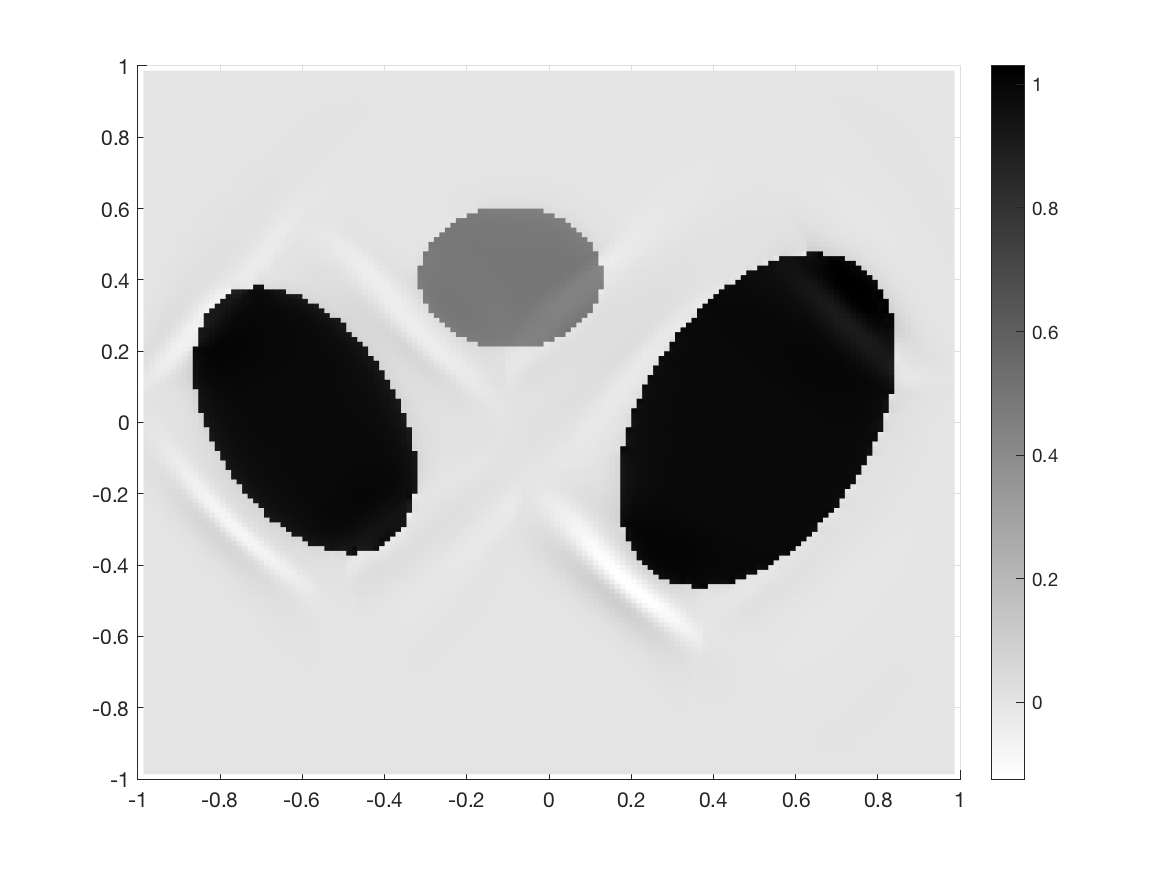}\label{heart_lung_recon3-01-01}}
\subfloat[ $\beta=0.3,\gamma=0.1, \delta=0.01$]{\includegraphics[width=0.35\textwidth]{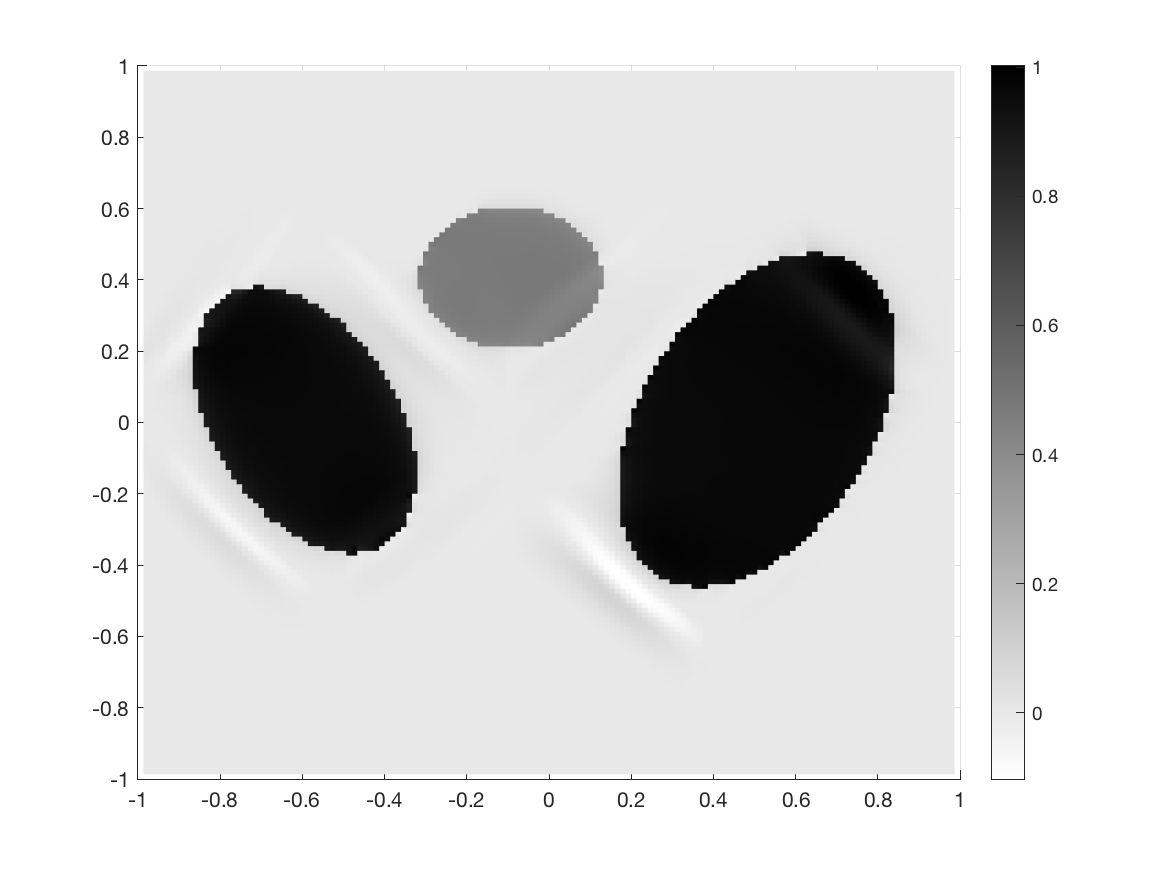}\label{heart_lung_recon3-1-01}}\\
\subfloat[ $\beta=0.3,\gamma=0.3,\delta=0.01$]{\includegraphics[width=0.35\textwidth]{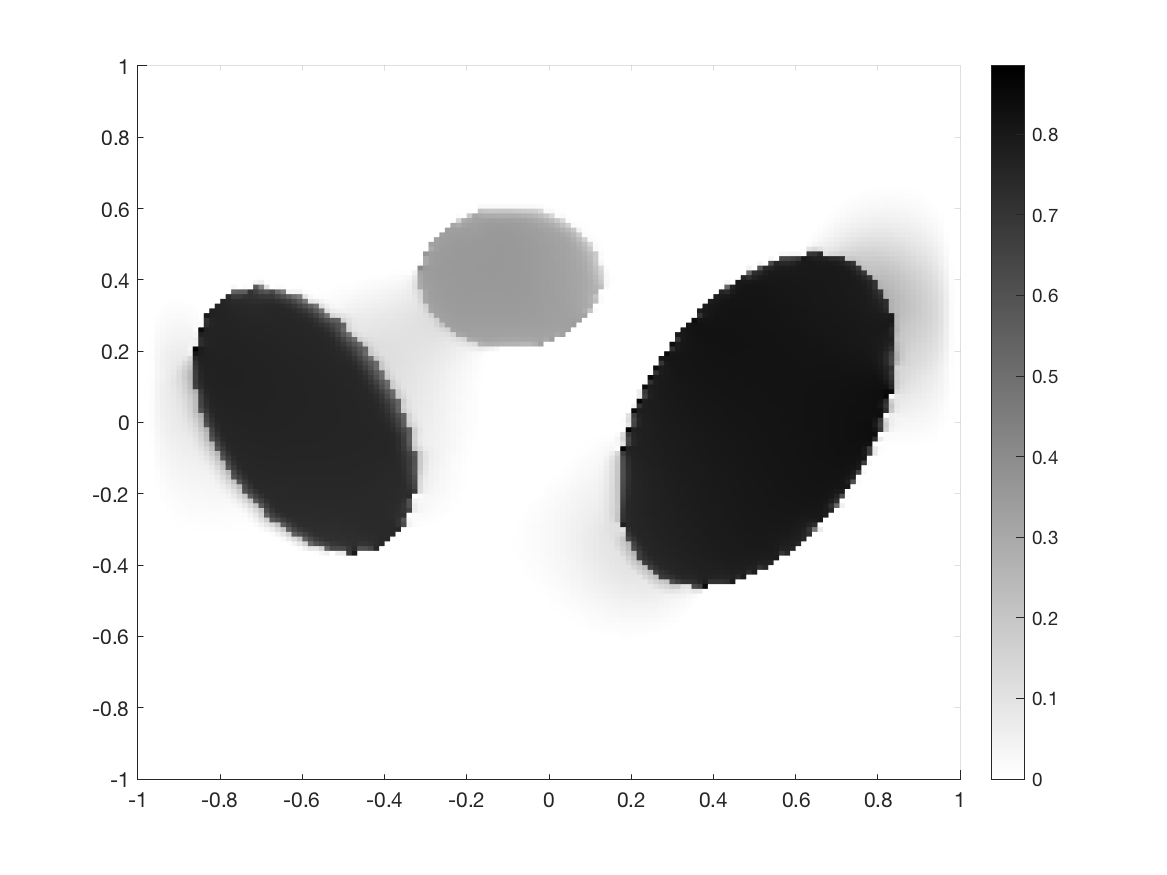}\label{heart_lung_recon3-3-01}}
\subfloat[ $\beta=0.3,\gamma=0.3,\delta=0.1$]{\includegraphics[width=0.35\textwidth]{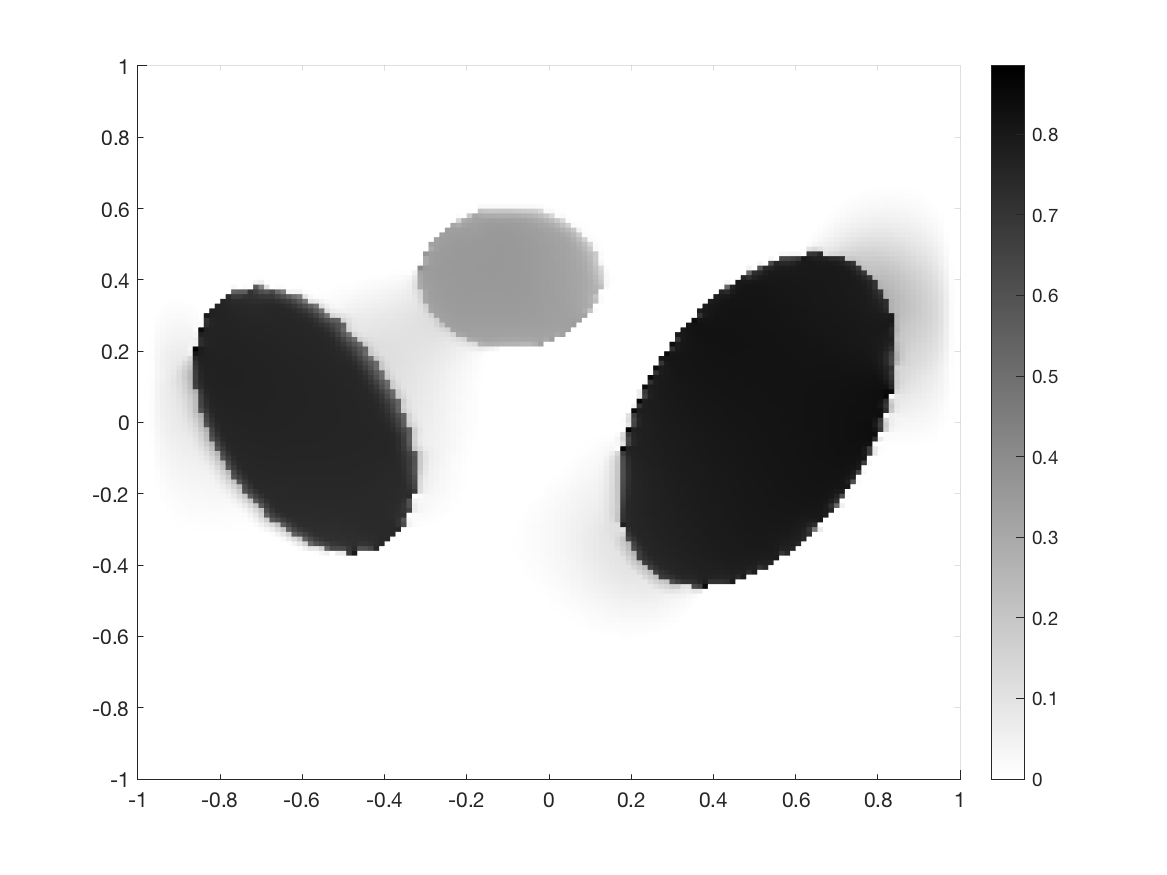}\label{heart_lung_recon3-3-1}}
\subfloat[ $\beta=0.03,\gamma=0.3,\delta=0.01$]{\includegraphics[width=0.35\textwidth]{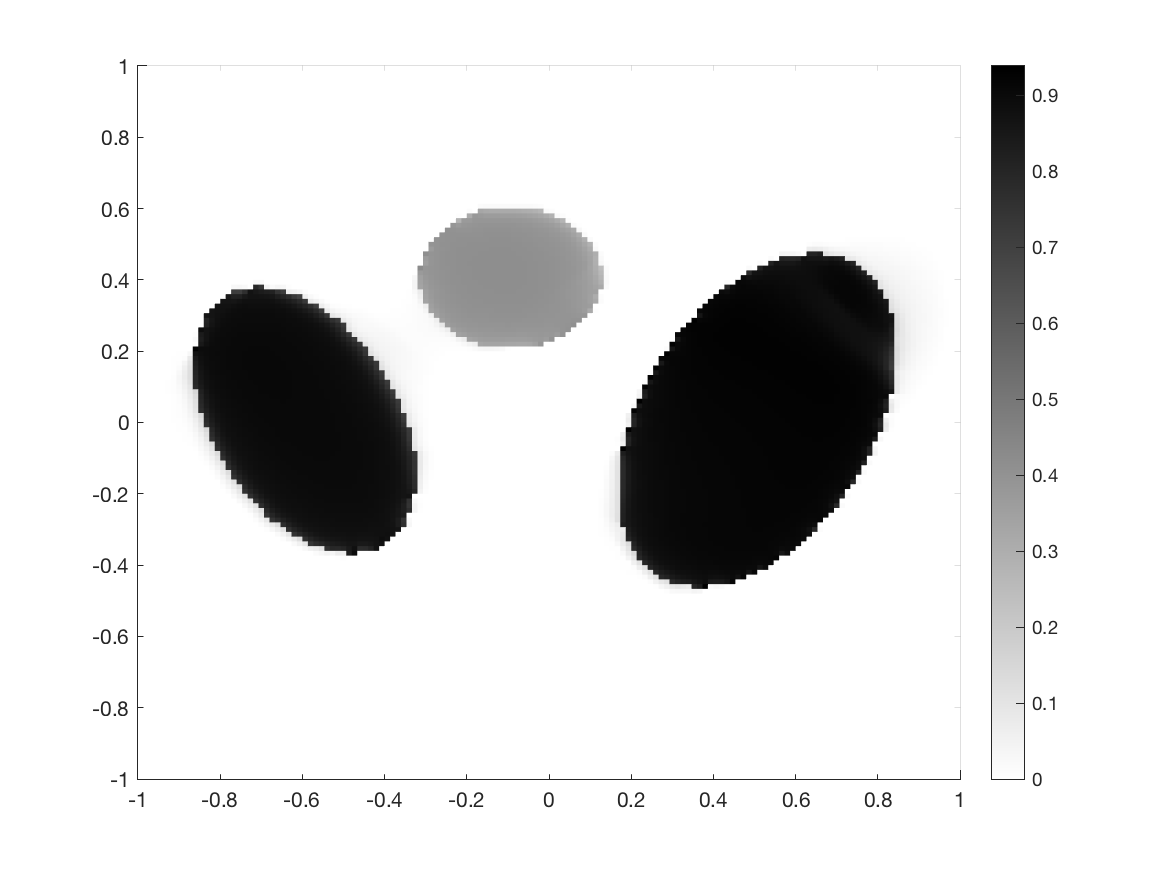}\label{heart_lung_recon03-3-01}}\\
\subfloat[ $\beta=0.03,\gamma=0.3,\delta=0.1$]{\includegraphics[width=0.35\textwidth]{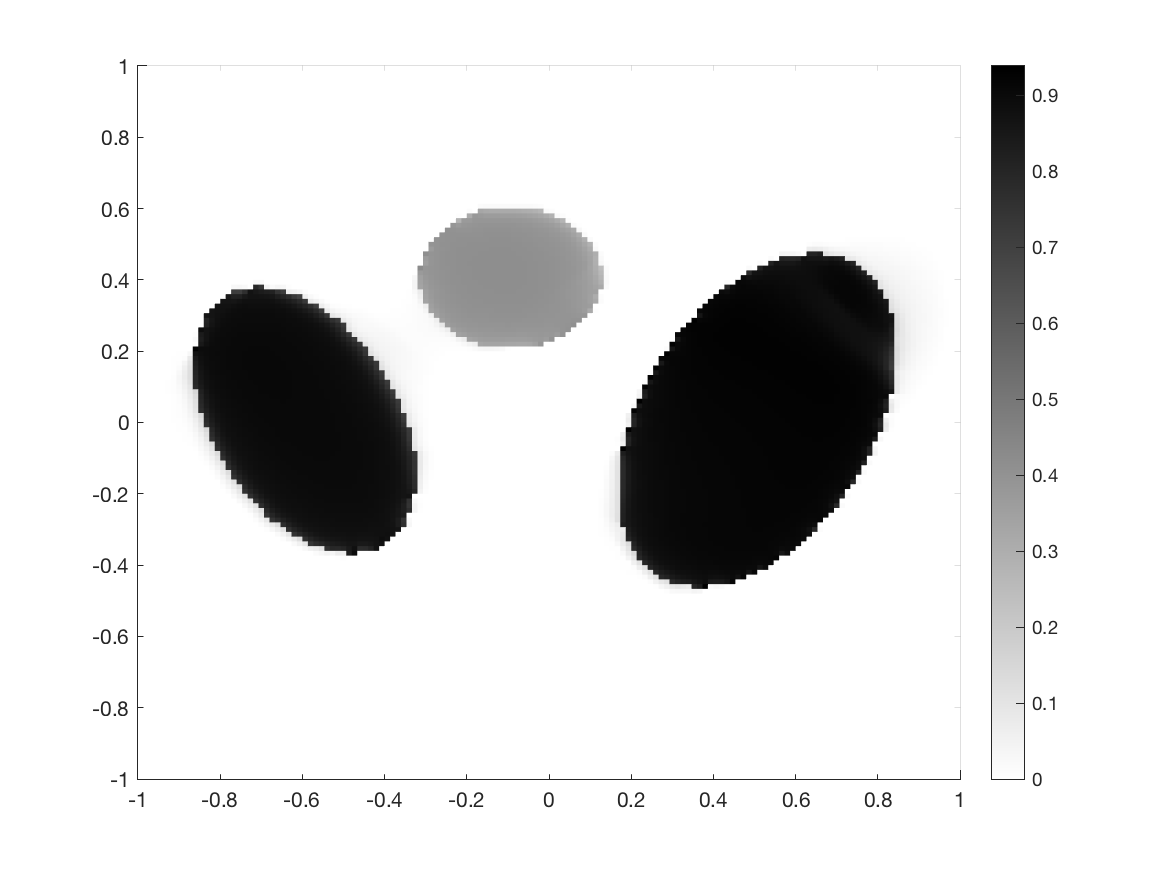}\label{heart_lung_recon03-3-1}}
\subfloat[Picard algorithm ]{\includegraphics[width=0.35\textwidth]{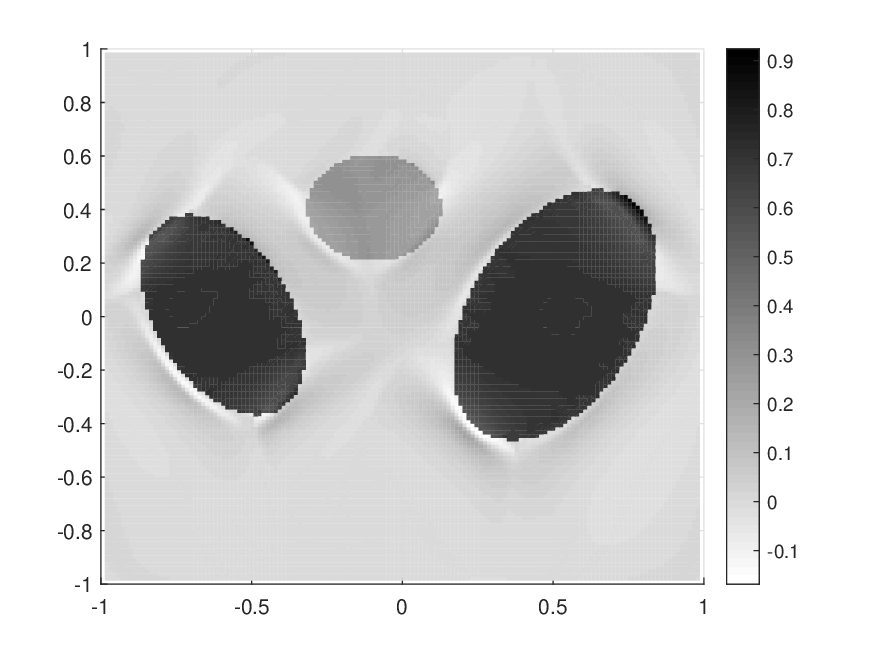}\label{heart_lung_recon_lin}}

    \caption{Test Case 2- The actual and reconstructed heart and lung phantom for various values of $\beta,\gamma,\delta$}
    \label{heart_lung}
  \end{figure}
  From Figure \ref{heart_lung_recon3-01-01}, we observe that with a small value of the $L^1$ regularization parameter $\gamma = 0.01$ in comparison to a higher $L^2$ regularization parameter $\beta = 0.3$, there are artifacts present in the image. Increasing the value of $\gamma$ to 0.1, keeping $\beta$ fixed results in a decrease of the artifacts as seen in Figure \ref{heart_lung_recon3-1-01}. A further increase in the value of $\gamma$ to 0.3 results in a significant reduction of artifacts as observed in Figure \ref{heart_lung_recon3-3-01}.  In Figure \ref{heart_lung_recon03-3-01},  we obtain a reconstruction with high resolution and contrast with a decrease in $\beta$ to 0.03. This is expected because the phantom under consideration in this example is sparse. Thus, with a greater weight for $L^1$ regularization term, coupled with the denoising operator controlled by the parameter $c$, described in Sec \ref{sec:VIP}, helps in removing artifacts and, thus, promotes high contrast.  We further note from Figures \ref{heart_lung_recon3-3-1} and \ref{heart_lung_recon03-3-1}  that changing the value of the Perona-Malik regularization term $\delta$ to 0.1 does not visibly change the resolution of the edges. This is because the term $\delta$ is linked to the denoising parameter $c$ that is fixed to be 0.001. For this choice of $c$, choosing $\delta = 0.01$ is sufficient to guarantee sharp edges.
 We further compare our results to the Picard algorithm proposed in \cite{Hoff2014}. The corresponding reconstruction is shown in Figure \ref{heart_lung_recon_lin}. It can be seen that the CDII-SR scheme provides a better contrast image, yet maintaining the same resolution as that of the Picard scheme. Also, there are far more artifacts through the Picard reconstruction method whereas the sparsity assumption and the $H^1$ denoising in CDII-SR scheme results in an image with very less artifacts.

 Further, to test the robustness of our method, we introduce $10\%$ and $25\%$ multiplicative Gaussian noises in the interior data $H^\delta$, which is fed as input to our CDII-SR algorithm. The corresponding reconstructions are shown in Figure \ref{heart_lung_recon_noise}. We also plot the reconstructions obtained with the Picard algorithm.

\begin{figure}[H]
\centering
\subfloat[ $\beta=0.03,\gamma=0.3,\delta=0.01$, \newline $10\%$ noise]{\includegraphics[width=0.4\textwidth]{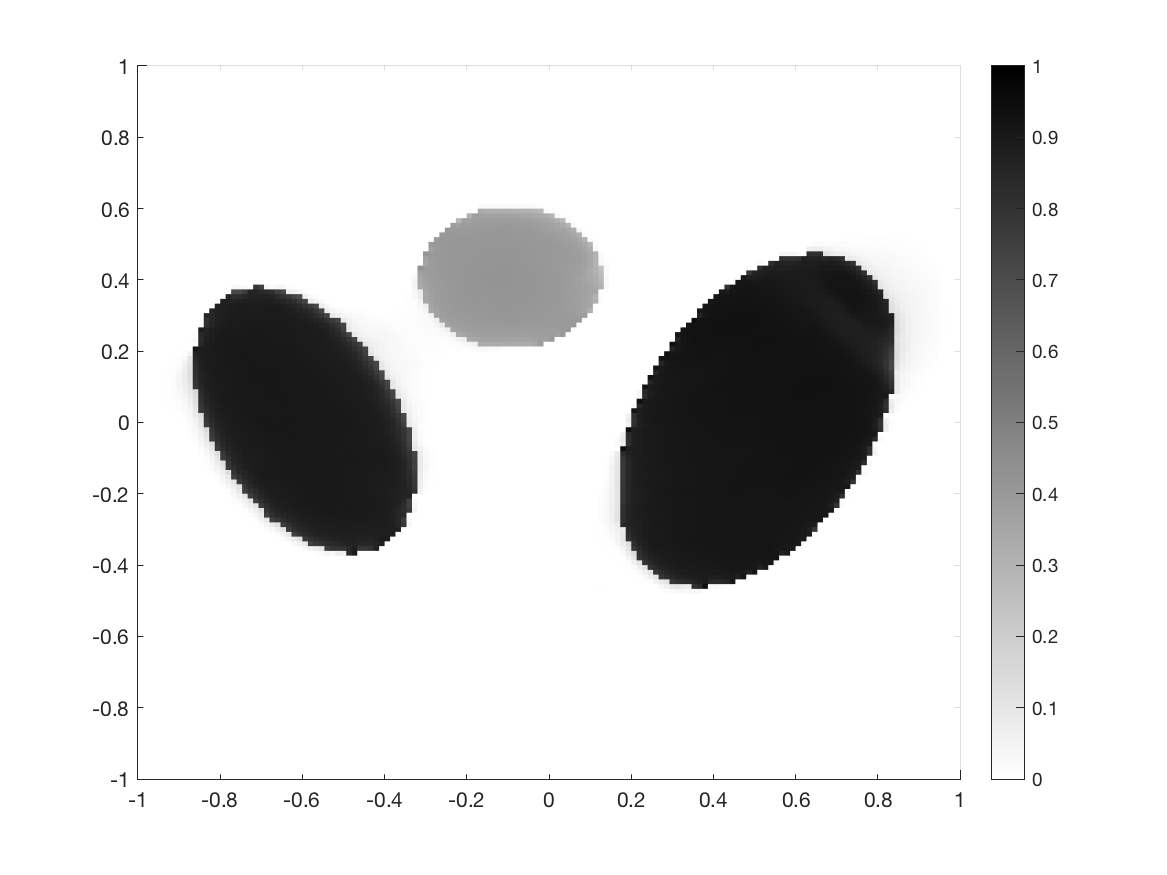}\label{heart_lung_recon03-3-01-10}}
\subfloat[ $\beta=0.05,\gamma=0.5,\delta=0.1$, \newline $25\%$ noise]{\includegraphics[width=0.4\textwidth]{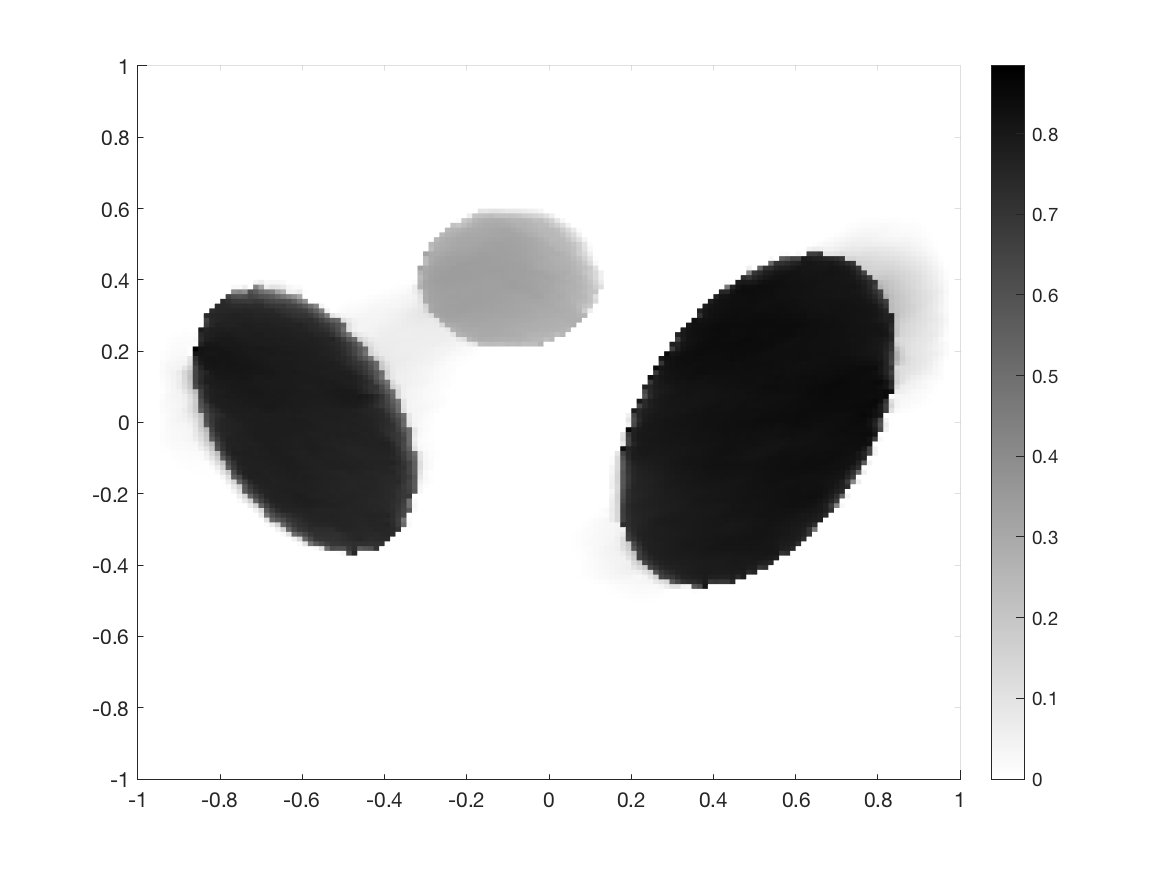}\label{heart_lung_recon03-3-01-25}}\\
\subfloat[Picard algorithm, $10\%$ noise]{\includegraphics[width=0.4\textwidth]{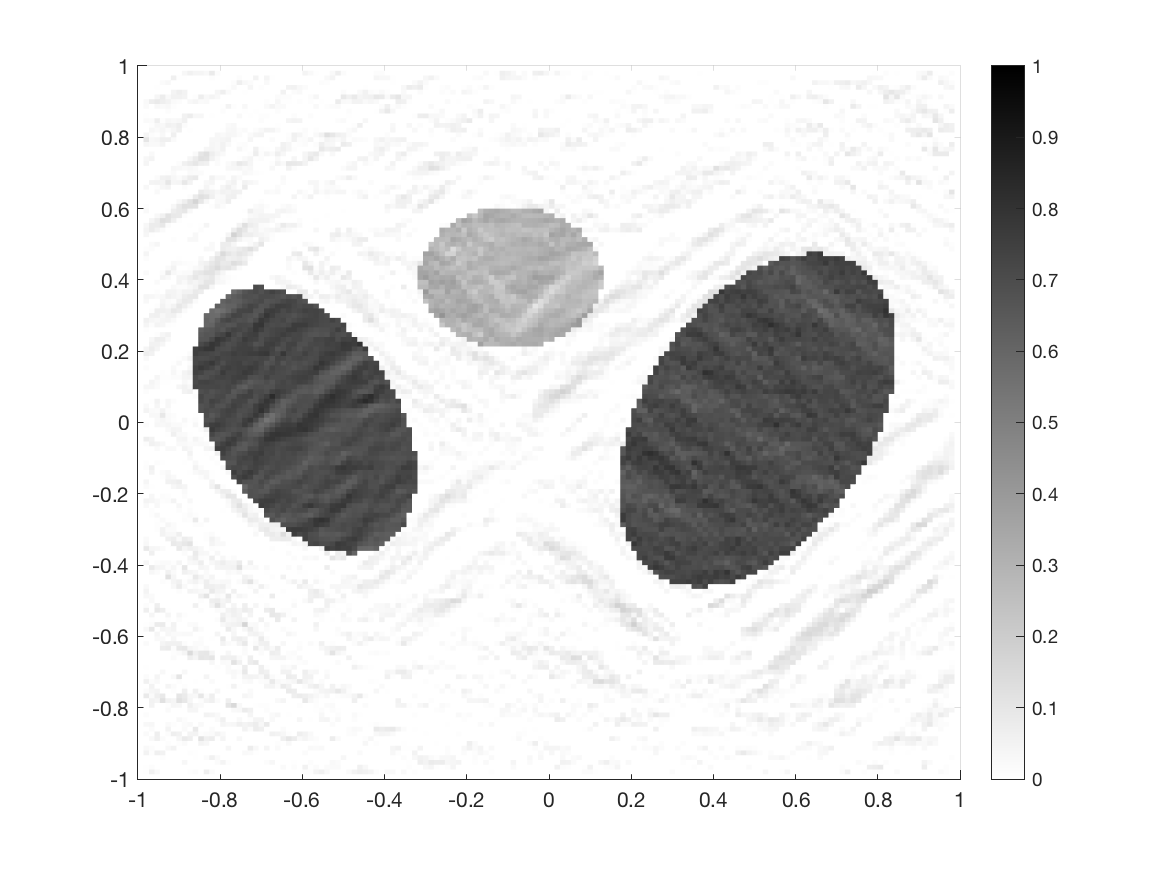}\label{heart_lung_recon_lin_10}}
\subfloat[Picard algorithm, $25\%$ noise]{\includegraphics[width=0.4\textwidth]{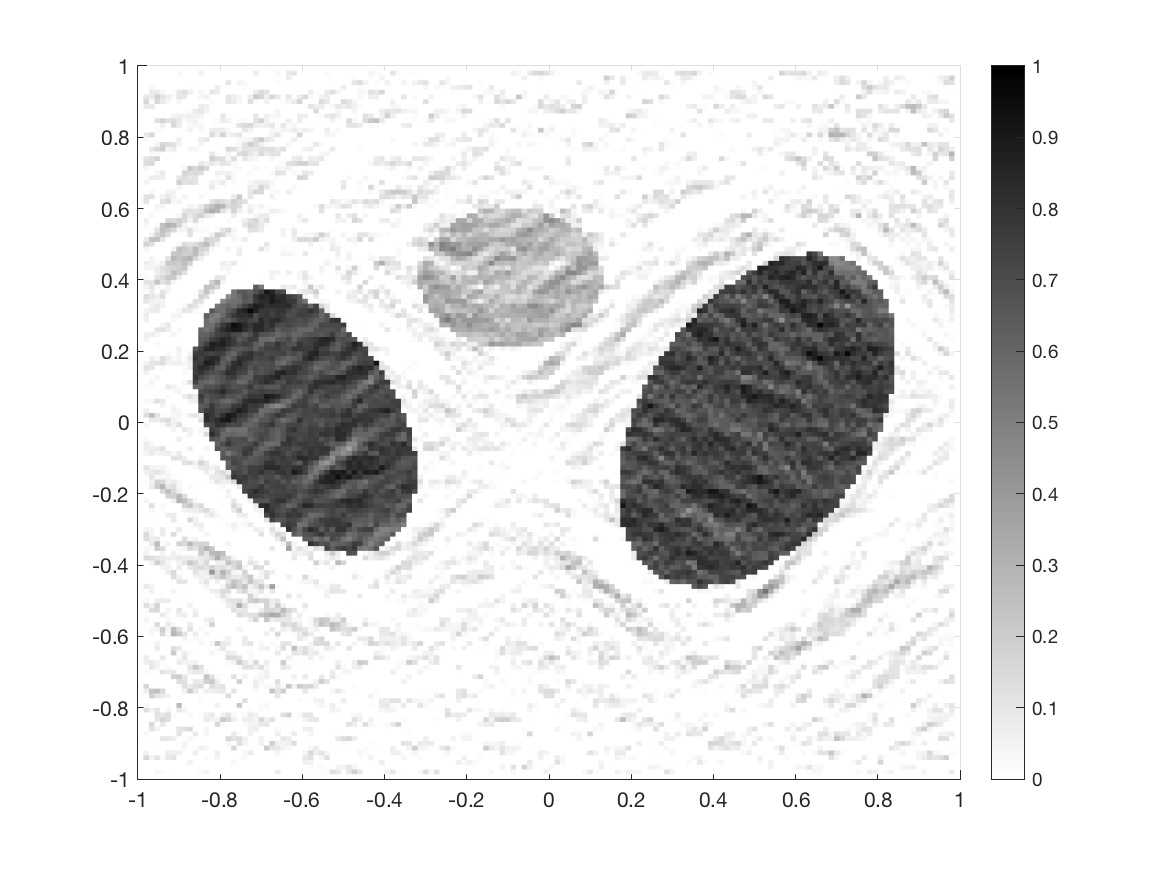}\label{heart_lung_recon_lin_25}}

    \caption{Test Case 2- The reconstructed heart and lung phantom for various values of noise in the interior data}
    \label{heart_lung_recon_noise}
  \end{figure}

We note from Figure \ref{heart_lung_recon03-3-01-10} that even with $10\%$ noise, we still have a high contrast and high resolution reconstruction, which demonstrates the robustness of our method. When the noise percentage was increased to $25\%$, it led to appearance of some streak artifacts, as seen in Figure \ref{heart_lung_recon03-3-01-25}. For this purpose, we used a higher value of $L^1$ regularization parameter $\gamma = 0.5$ and denoising parameter $c=0.01$ to remove artifacts. Due to a higher value of $c$, we increase the value of $\delta=0.1$. We also used the Picard algorithm with the noisy data and the reconstructions showed visible presence of streak artifacts in Figure \ref{heart_lung_recon_lin_10} and \ref{heart_lung_recon_lin_25}.

In Test Case 3, we consider a combination of phantoms, where one is supported on a square annulus $S_a = \lbrace (x,y) \in \mathbb{R}^2: -0.8 < x < -0.7, -0.2 < x < -0.1,-0.8 < y < -0.7, -0.2 < y < -0.1  \rbrace $ with $\sigma=3.0$; 
the other one consists of 2 disks centered at $(0.7,0.7)$ with radius 0.2 and $\sigma=1.0$ and at $(0.55,0.55)$ with radius 0.15 and $\sigma=2.0$. The value of $\sigma$ inside the square annulus has a value -2.0. The plots of the actual and the reconstructed $\sigma$ are shown in Figure \ref{mixed}.

\begin{figure}[H]
\centering
\subfloat[Actual phantom]{\includegraphics[width=0.35\textwidth]{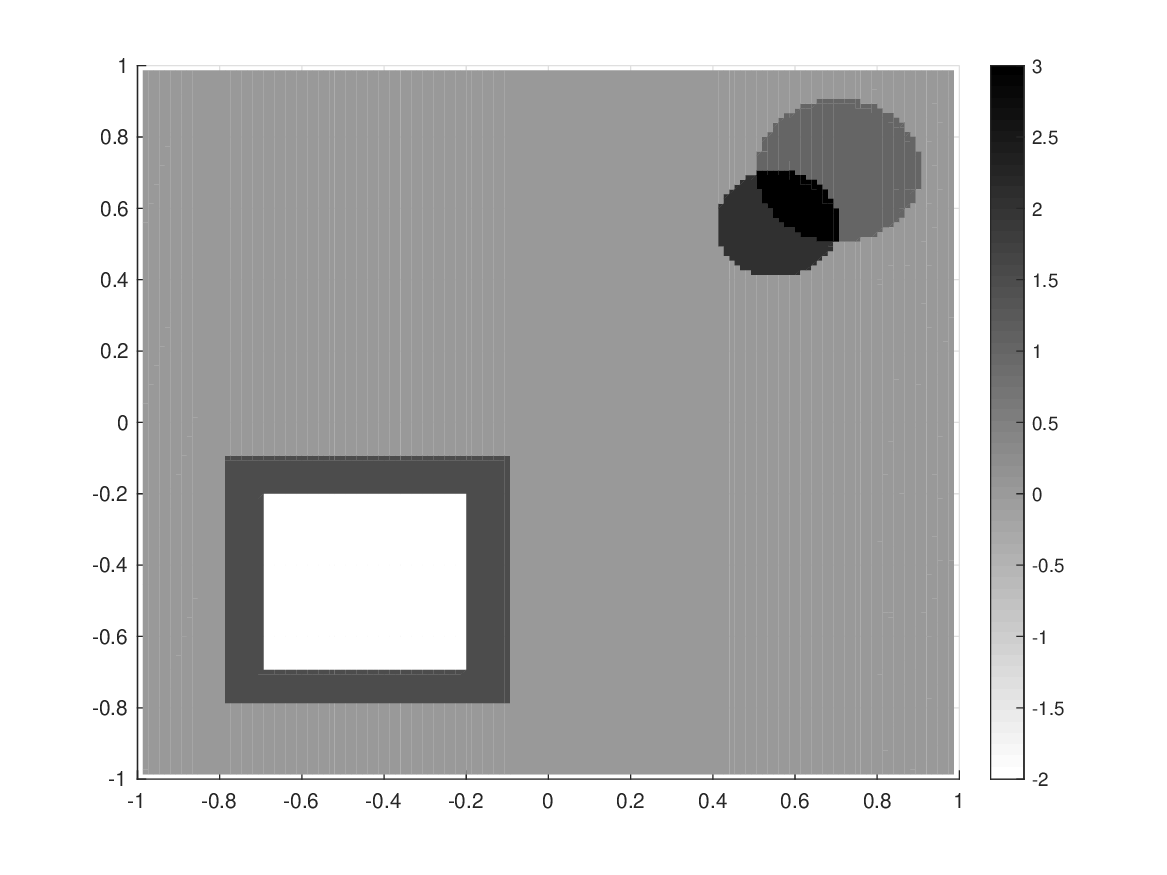}\label{mixed_actual}}
\subfloat[ $\beta=0.3,\gamma=0.01,\delta=0.01$]{\includegraphics[width=0.35\textwidth]{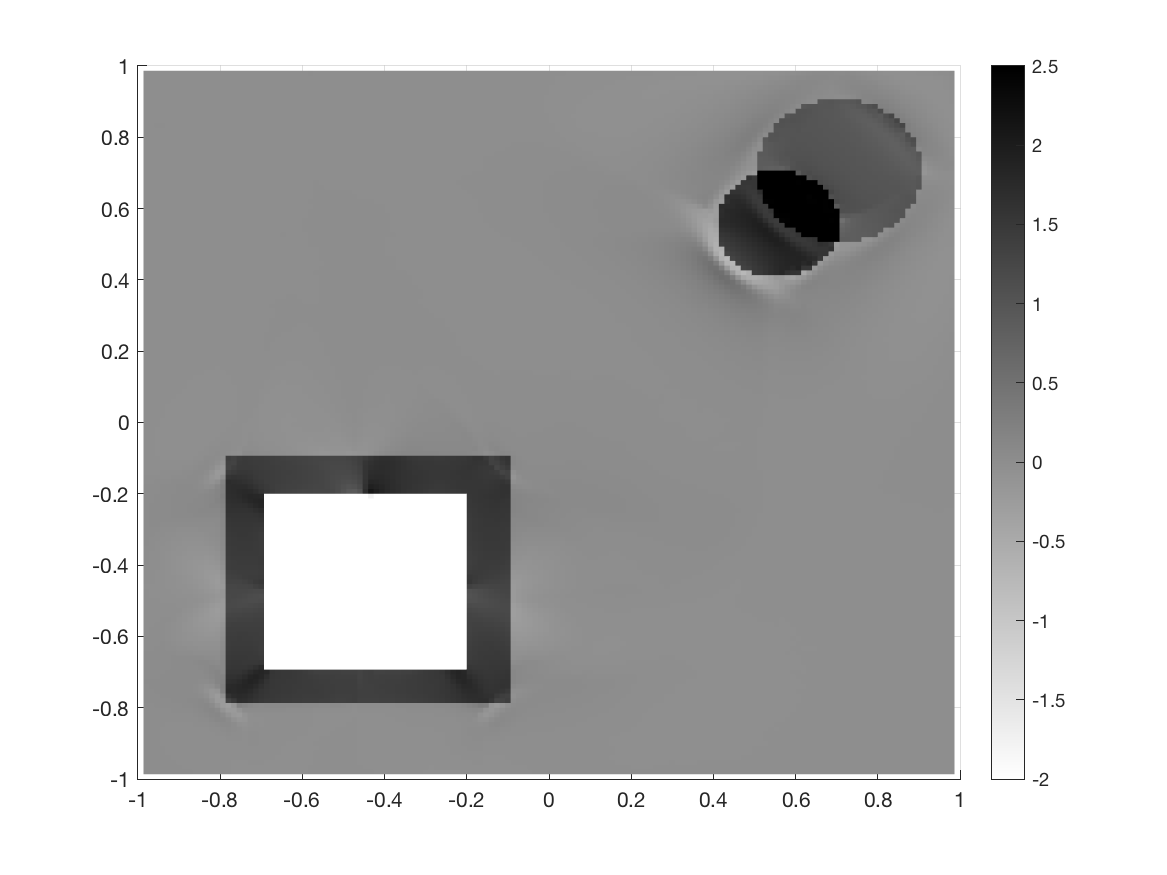}\label{mixed_recon3-01-01}}
\subfloat[ $\beta=0.3,\gamma=0.1, \delta=0.01$]{\includegraphics[width=0.35\textwidth]{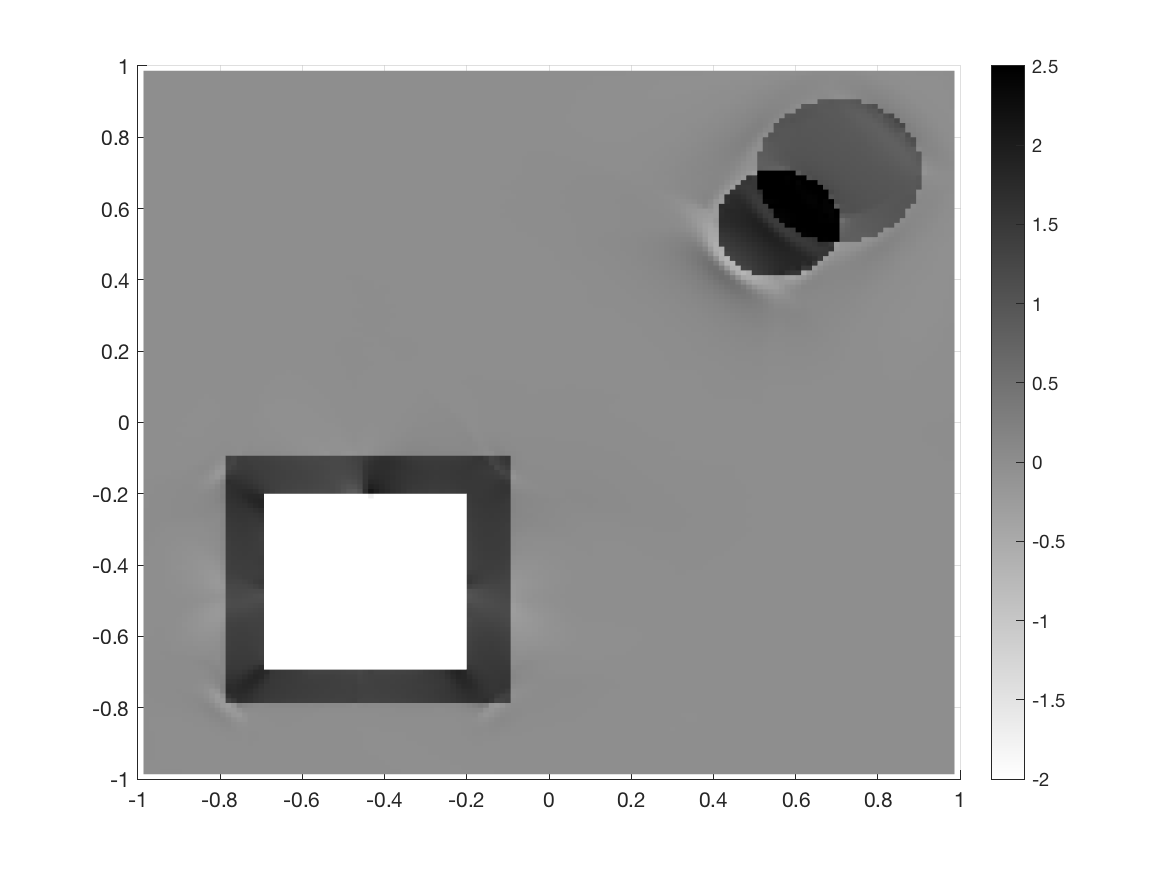}\label{mixed_recon3-1-01}}\\
\subfloat[ $\beta=0.3,\gamma=0.3,\delta=0.01$]{\includegraphics[width=0.35\textwidth]{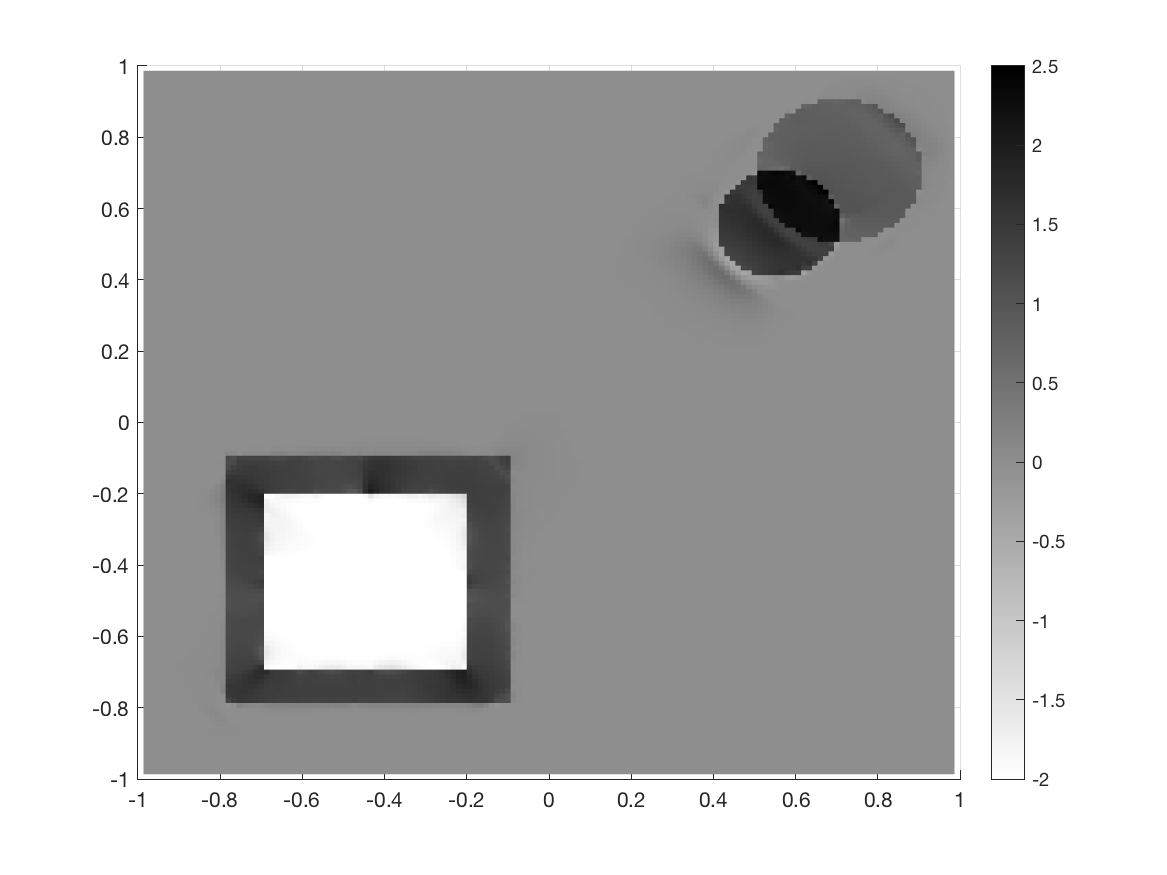}\label{mixed_recon3-3-01}}
\subfloat[ $\beta=0.3,\gamma=0.3,\delta=0.1$]{\includegraphics[width=0.35\textwidth]{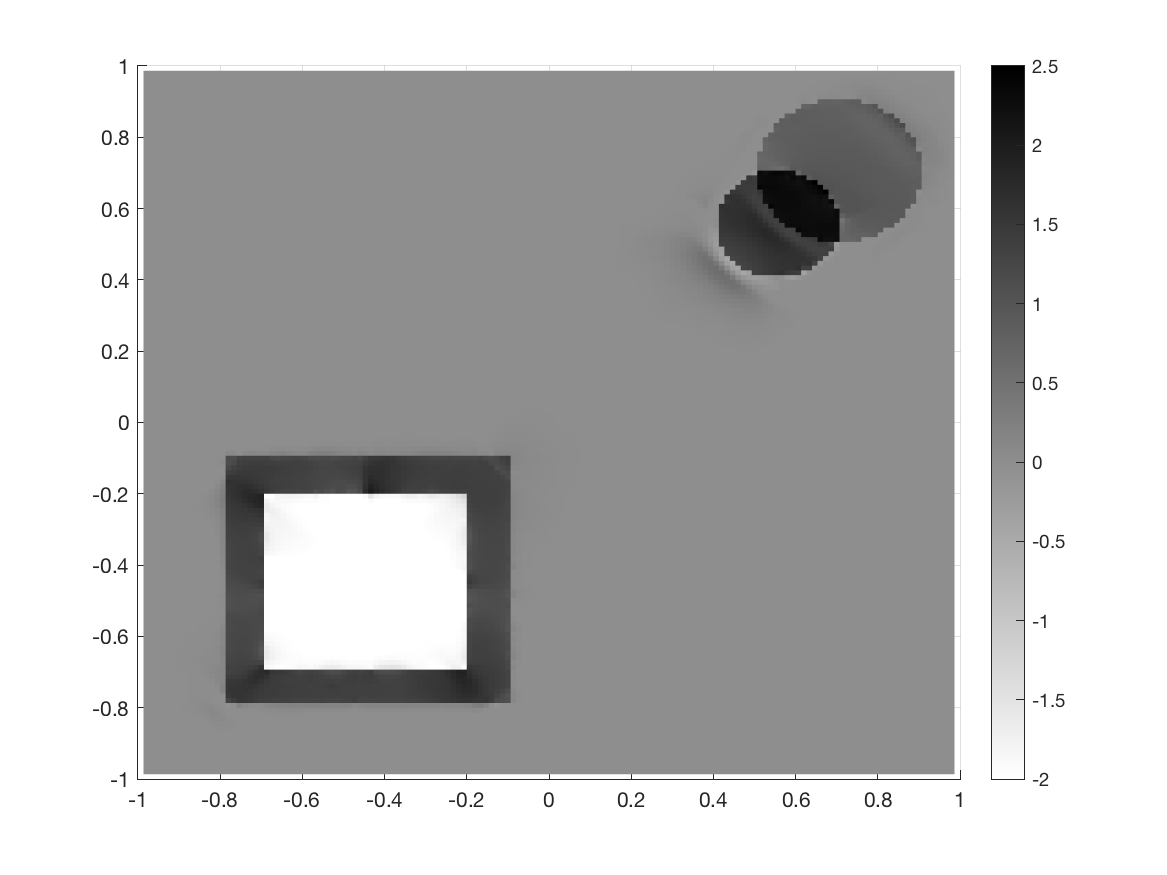}\label{mixed_recon3-3-1}}
\subfloat[ $\beta=0.03,\gamma=0.3,\delta=0.01$]{\includegraphics[width=0.35\textwidth]{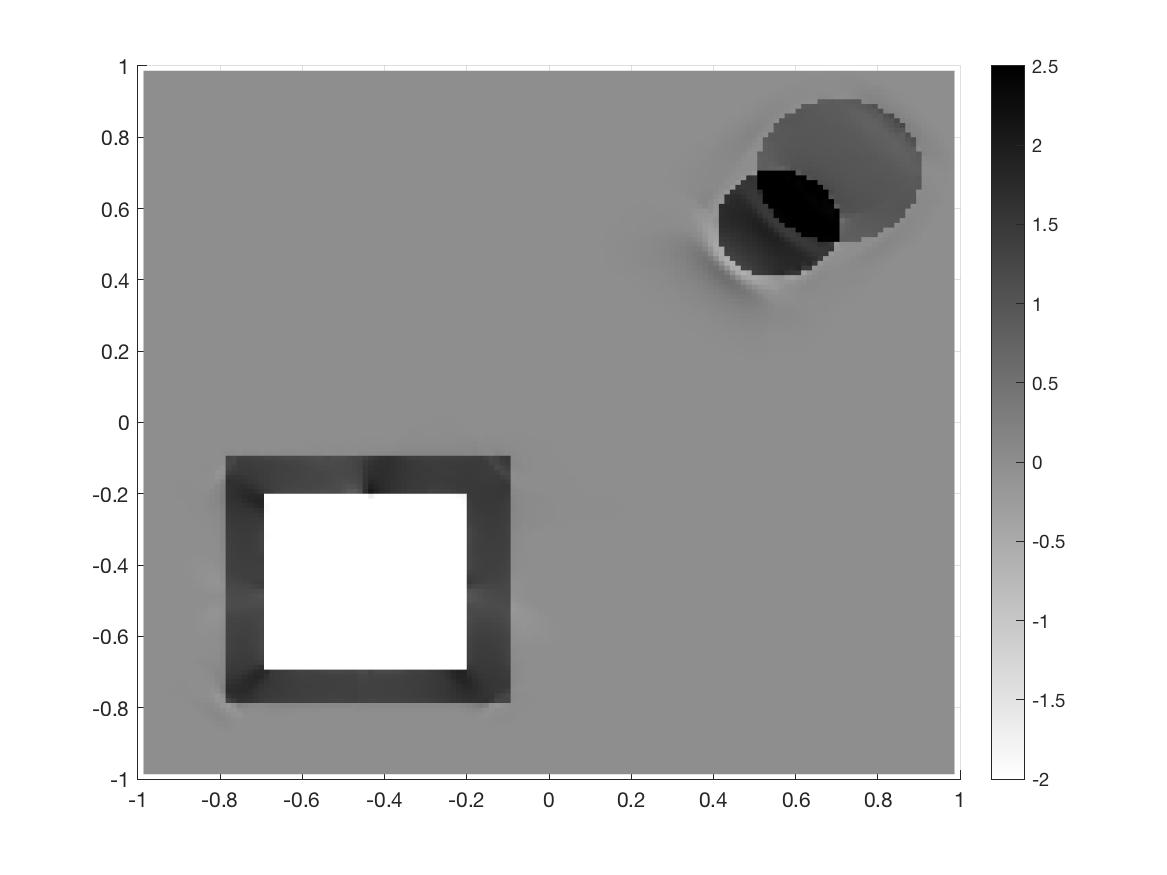}\label{mixed_recon03-3-01}}\\
\subfloat[ $\beta=0.03,\gamma=0.3,\delta=0.1$]{\includegraphics[width=0.35\textwidth]{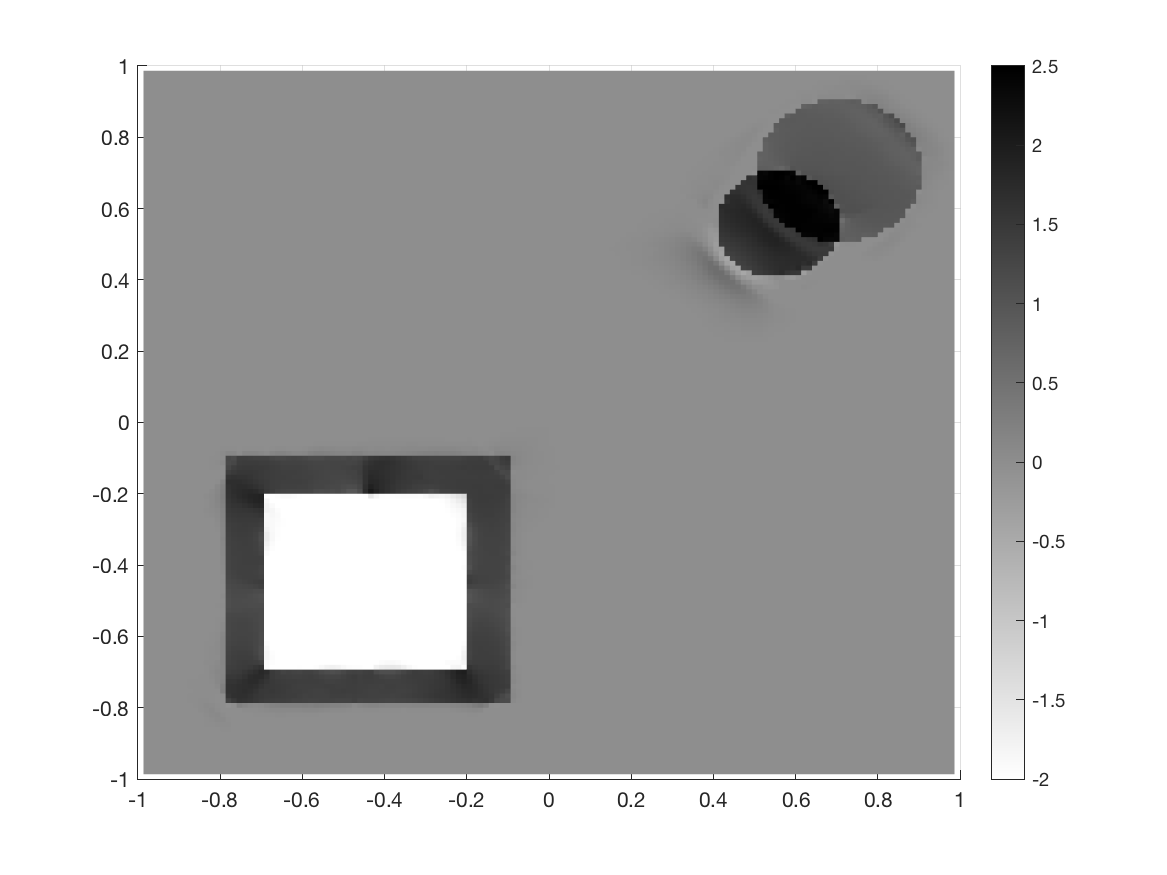}\label{mixed_recon03-3-1}}
\subfloat[Picard algorithm ]{\includegraphics[width=0.35\textwidth]{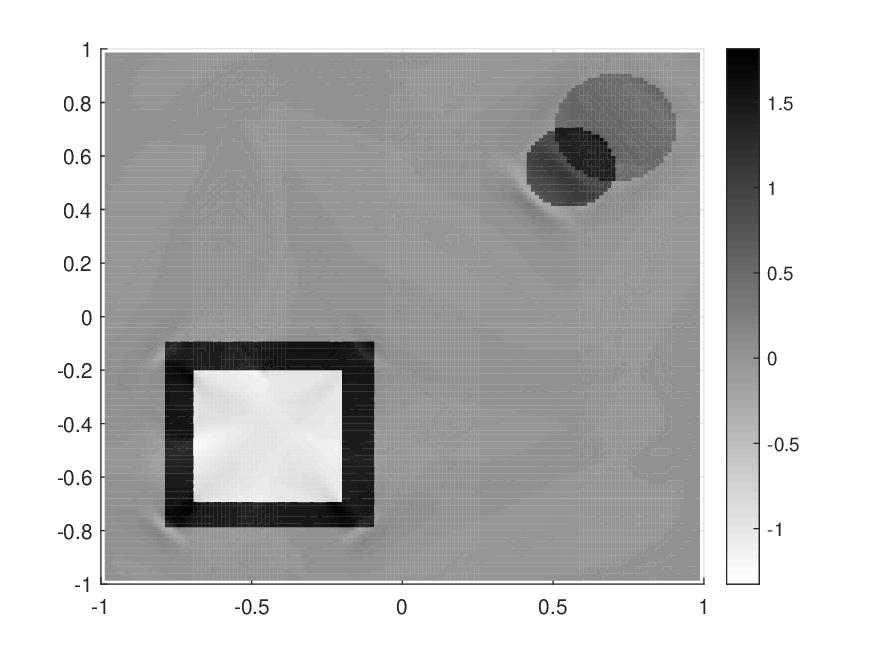}\label{mixed_recon_lin}}

    \caption{Test Case 3- The actual and reconstructed mixed phantom for various values of $\beta,\gamma,\delta$}
    \label{mixed}
  \end{figure}
  
We again observe a similar behavior as in Test Case 2. For the values of $\beta = 0.3$ and $\gamma = 0.01, 0.1$, we obtain reconstructions that have visible artifacts, especially at the corners of the square annulus. This can be seen in Figure \ref{mixed_recon3-01-01} and \ref{mixed_recon3-1-01}.  When $\gamma$ is increased to 0.3 and when $\beta$ is subsequently lowered to 0.03, we have almost no artifacts except for some in front of the smaller disk as seen in Figure \ref{mixed_recon03-3-01} and \ref{mixed_recon03-3-1}. We also observe that the values inside the annular hole and the inclusion due to the intersection of the two disks are very well reconstructed with nice contrast and high resolution. With the Picard method, we obtain a large number of artifacts and the value inside the hole is also not reconstructed well, thus, leading to loss of contrast as seen in Figure \ref{mixed_recon_lin}.
 
We also test the performance of our method in presence of $10\%$ and $25\%$ Gaussian noise in the interior data. The results are shown in Figure \ref{mixed_recon_noise}. 
\begin{figure}[H]
\centering
\subfloat[ $\beta=0.03,\gamma=0.3,\delta=0.01$, \newline $10\%$ noise]{\includegraphics[width=0.4\textwidth]{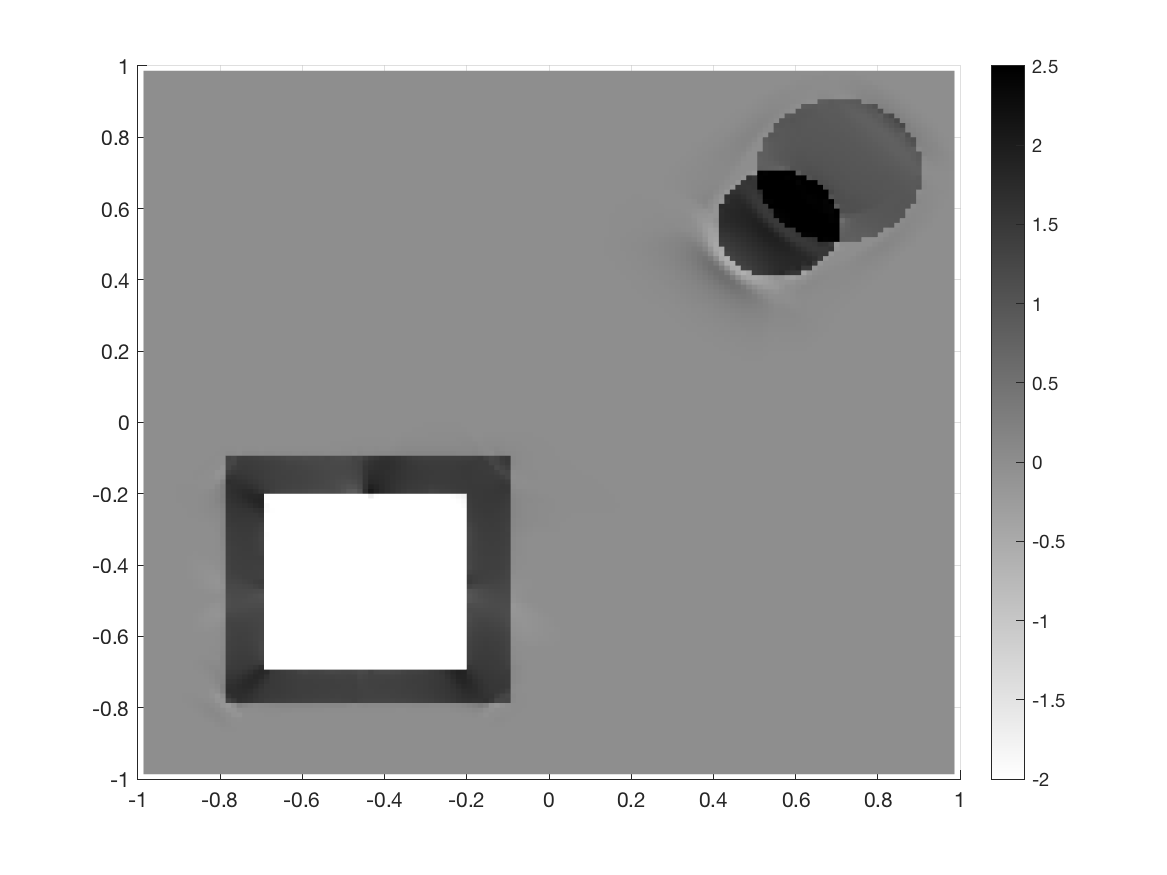}\label{mixed_recon03-3-01-10}}
\subfloat[ $\beta=0.05,\gamma=0.5,\delta=0.1$, \newline $25\%$ noise]{\includegraphics[width=0.4\textwidth]{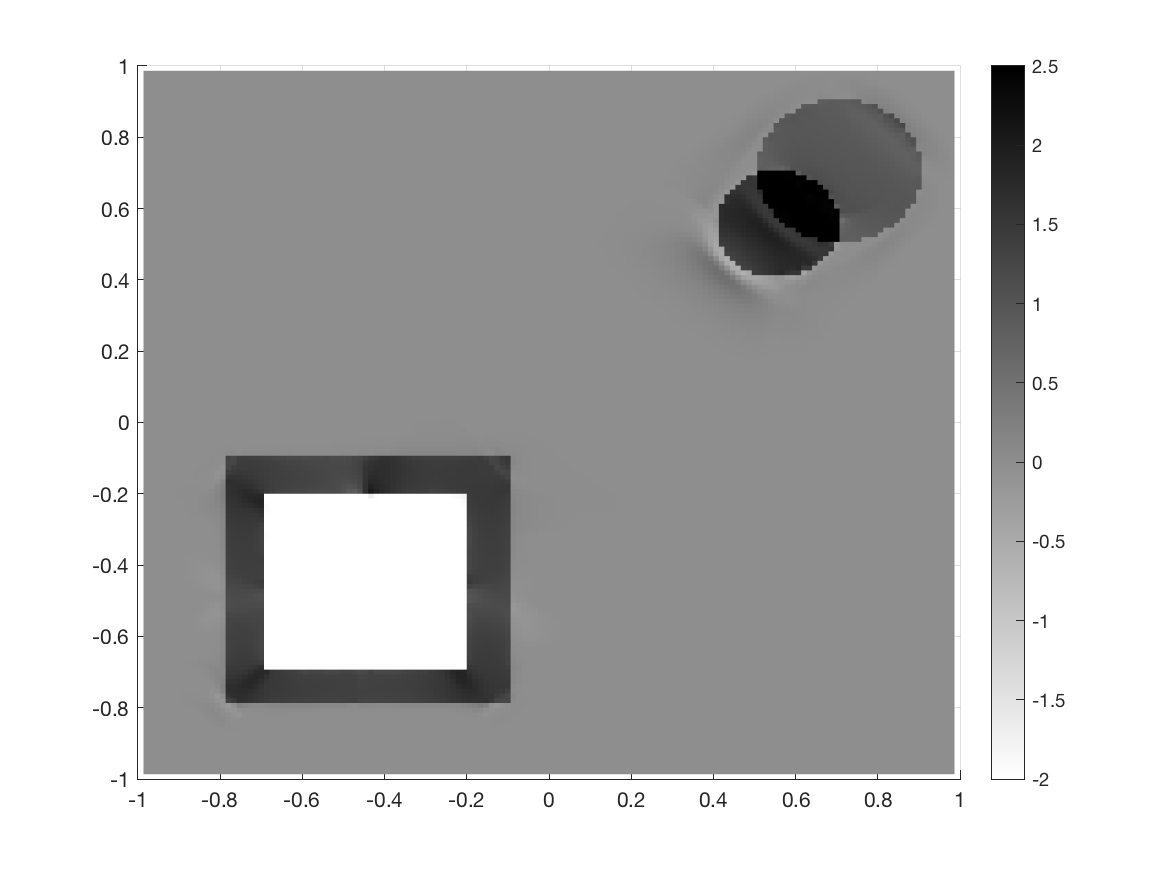}\label{mixed_recon05-5-01-25}}\\
\subfloat[Picard algorithm, $10\%$ noise]{\includegraphics[width=0.4\textwidth]{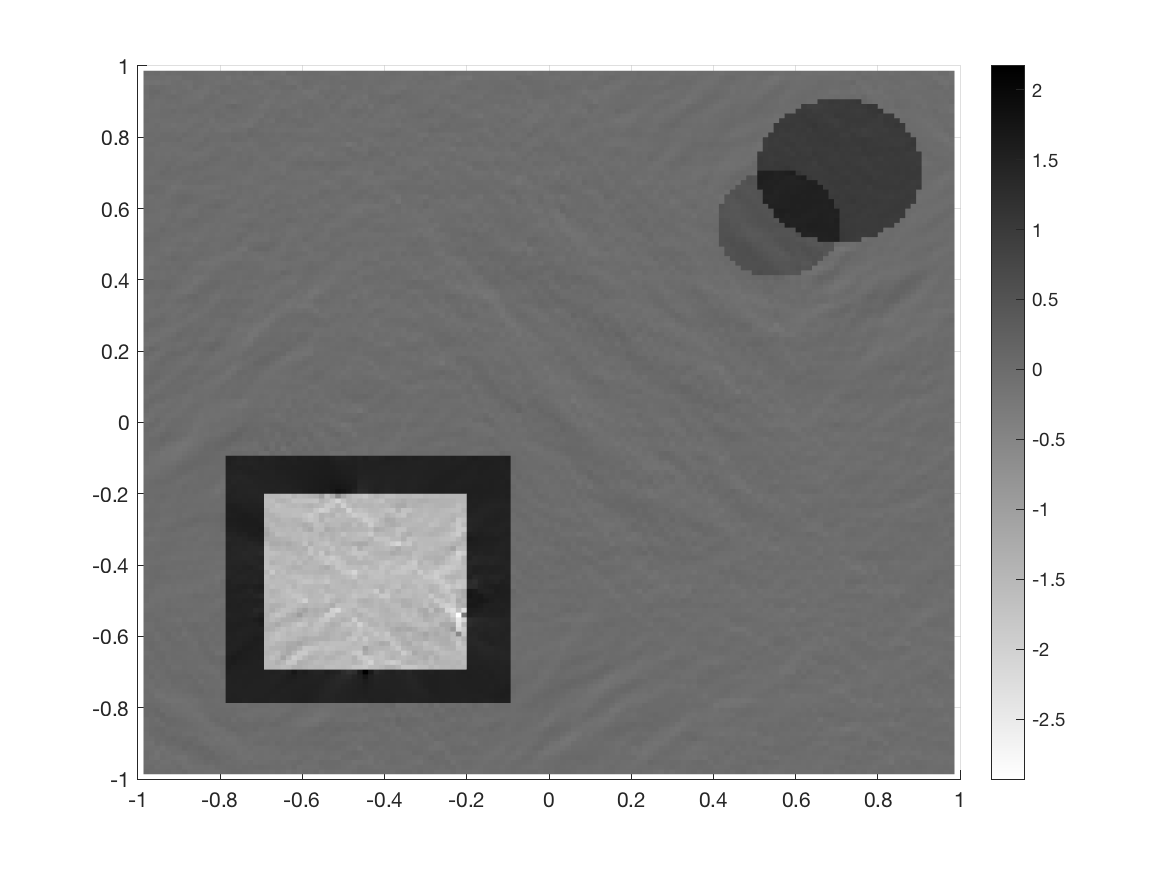}\label{mixed_recon_lin_10}}
\subfloat[Picard algorithm, $25\%$ noise]{\includegraphics[width=0.4\textwidth]{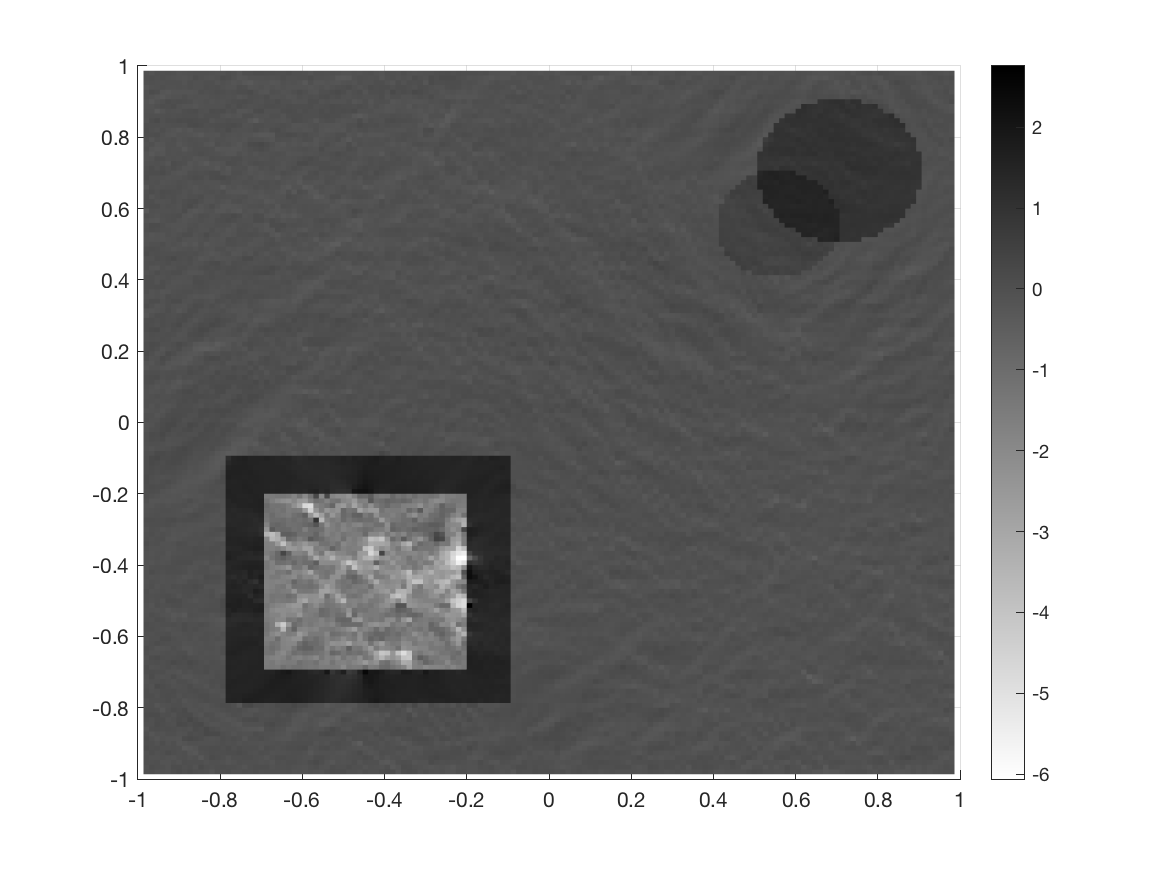}\label{mixed_recon_lin_25}}

    \caption{Test Case 3- The reconstructed mixed phantom for various values of noise in the interior data}
    \label{mixed_recon_noise}
  \end{figure}
  
  We again observe that our method gives superior quality reconstructions as shown in Figures \ref{mixed_recon03-3-01-10} and \ref{mixed_recon05-5-01-25}. As with Test Case 2, we need greater values of $\beta,\gamma,\delta,c$ to obtain high contrast and high resolution images. From Figures \ref{mixed_recon_lin_10} and \ref{mixed_recon_lin_25}, we see that, in comparison to the reconstructions with the CDII-SR scheme, the Picard scheme does not provide good reconstructions of objects with holes and inclusions in presence of noise. In particular, the region inside the hole contains a large number of artifacts. This shows the robustness and efficiency of our method in the reocnstructing objects with holes and inclusions.

In this final Test Case 4, we consider another combination of phantoms that result in a lesser sparsity pattern. This experiment is conducted to demonstrate the effect of the $L^1$ regularization term in improving contrast. The first phantom is supported on a square annulus $S_a = \lbrace (x,y) \in \mathbb{R}^2: -0.8 < x < -0.7, -0.2 < x < -0.1,-0.8 < y < -0.7, -0.2 < y < -0.1  \rbrace $ with $\sigma=3.0$. The value of $\sigma$ inside the square annulus has a value -1.5. The second phantom consists of 4 disks: the first centered at $(0.7,0.7)$ with radius 0.2 and $\sigma=1.0$, the second centered at $(0.55,0.55)$ with radius 0.15 and $\sigma=2.0$, the third centered at $(0,0)$ with radius 0.25 and $\sigma=1.5$ and the last centered at $(0.05,0.6)$ with radius 0.2 and $\sigma=2.5$. The final phantom consists of two heart-shaped objects, represented by cardioids, with the value of $\sigma=4$ in the bigger one and value of $\sigma=3.0$ in the smaller one.  The plots of the actual and the reconstructed $\sigma$ are shown in Figure \ref{mixed2}.

\begin{figure}[H]
\centering
\subfloat[Actual phantom]{\includegraphics[width=0.35\textwidth]{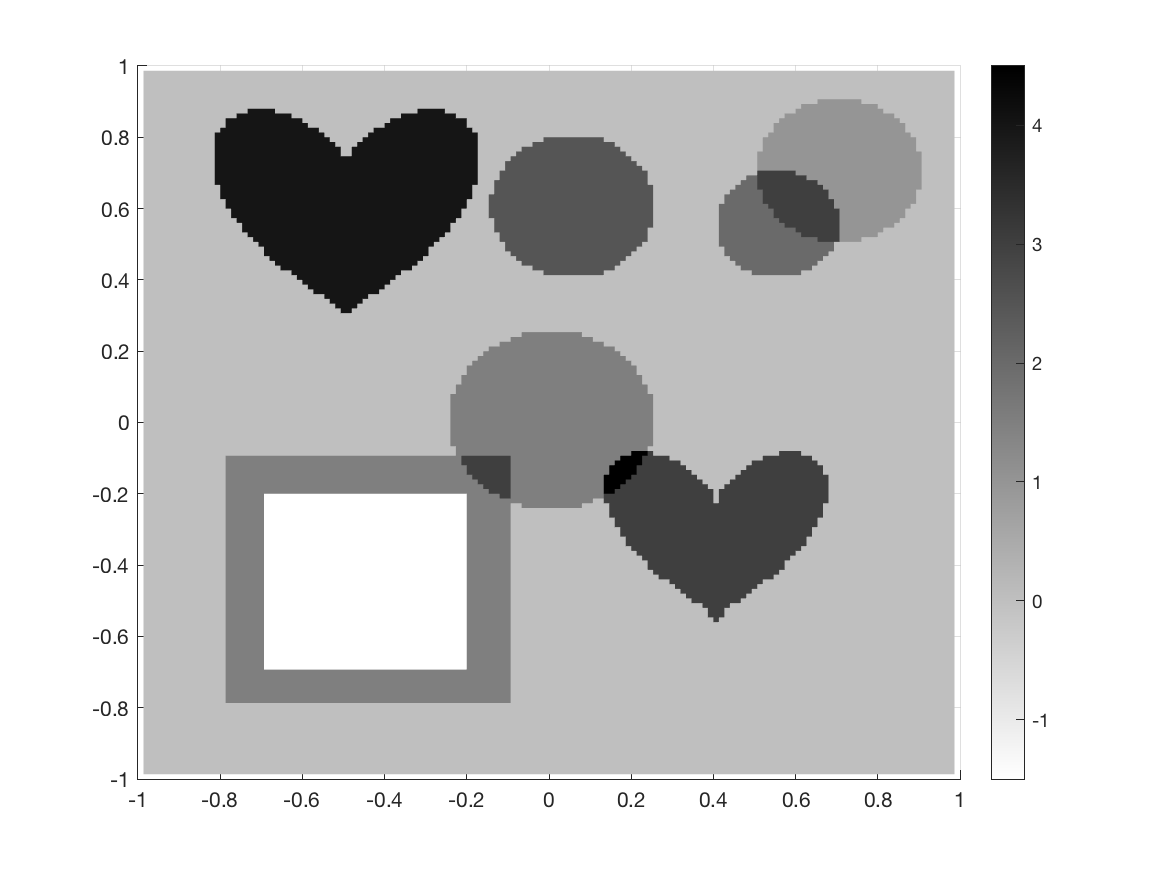}\label{mixed2_actual}}
\subfloat[ $\beta=0.3,\gamma=0.01,\delta=0.01$]{\includegraphics[width=0.35\textwidth]{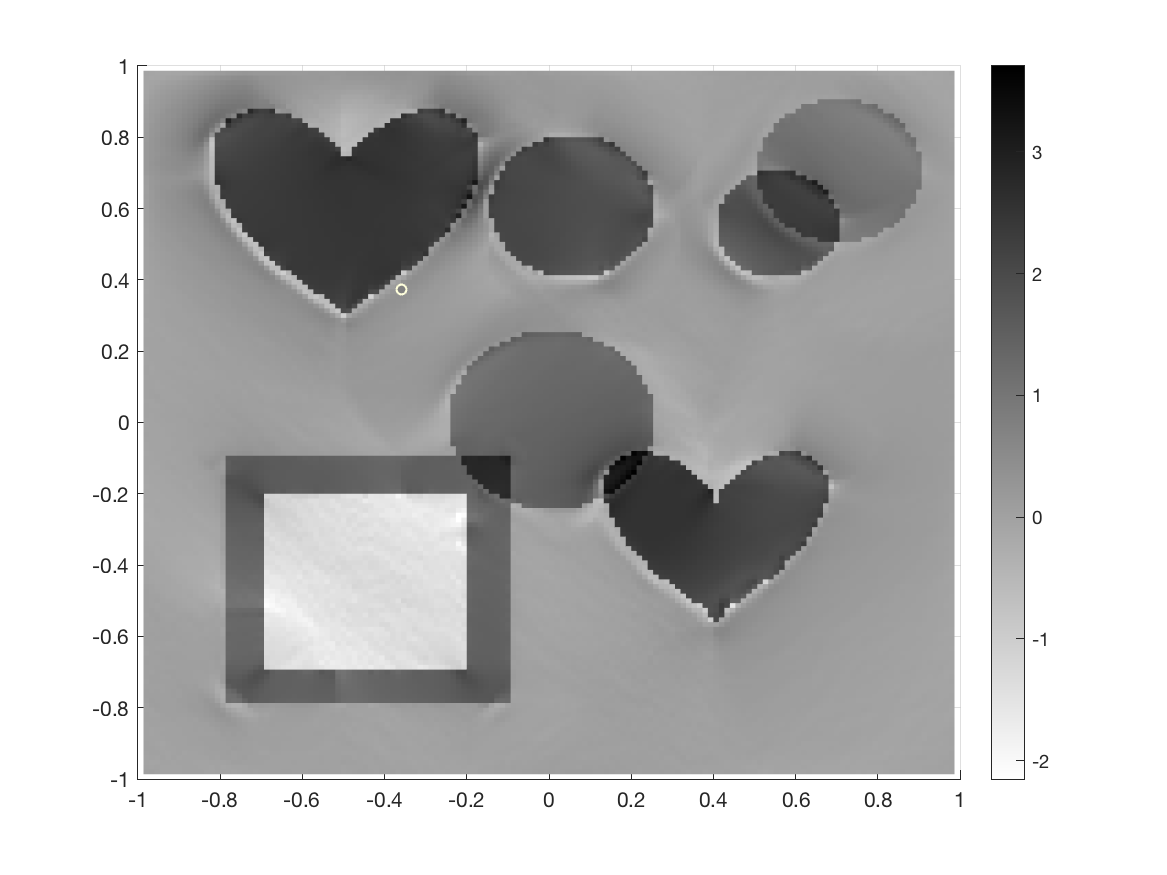}\label{mixed2_recon3-01-01}}
\subfloat[ $\beta=0.3,\gamma=0.3, \delta=0.01$]{\includegraphics[width=0.35\textwidth]{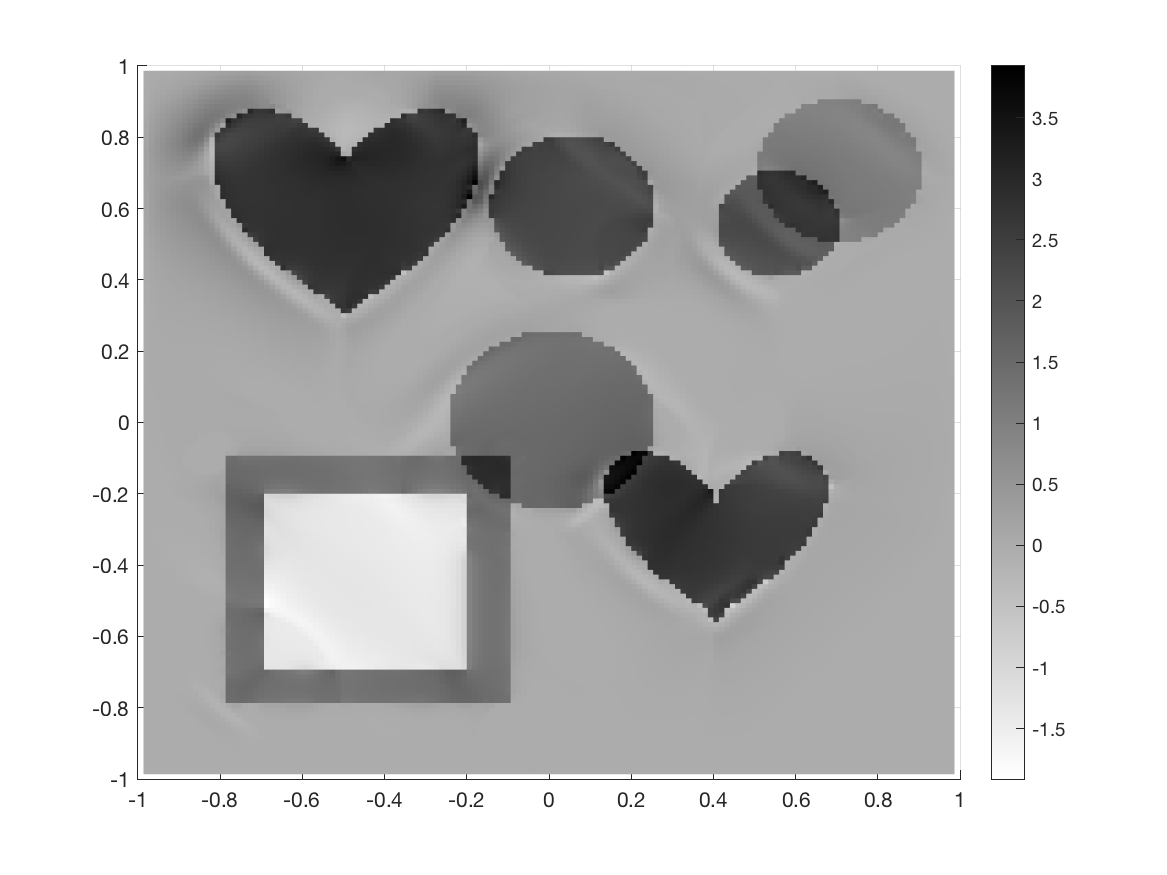}\label{mixed2_recon3-3-01}}\\
\subfloat[ $\beta=0.3,\gamma=0.5,\delta=0.01$]{\includegraphics[width=0.35\textwidth]{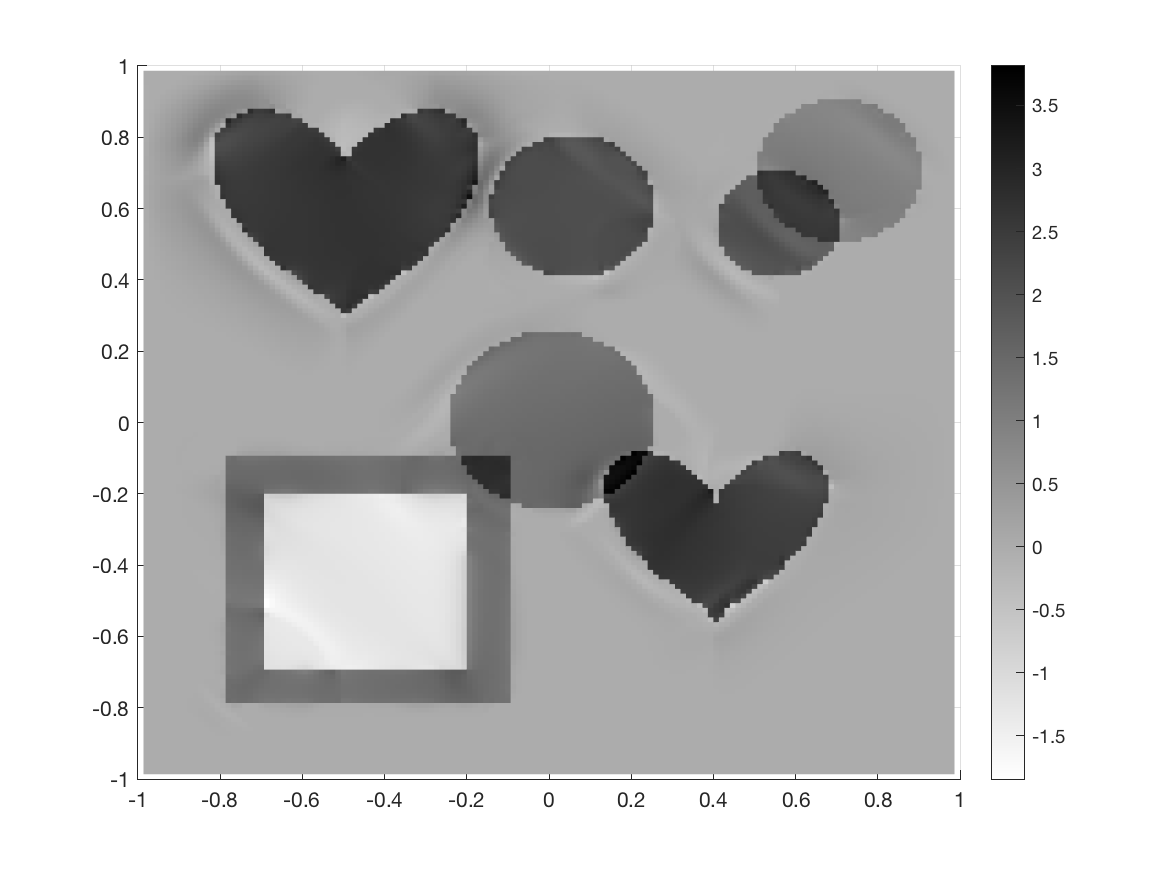}\label{mixed2_recon3-5-01}}
\subfloat[Picard algorithm ]{\includegraphics[width=0.35\textwidth]{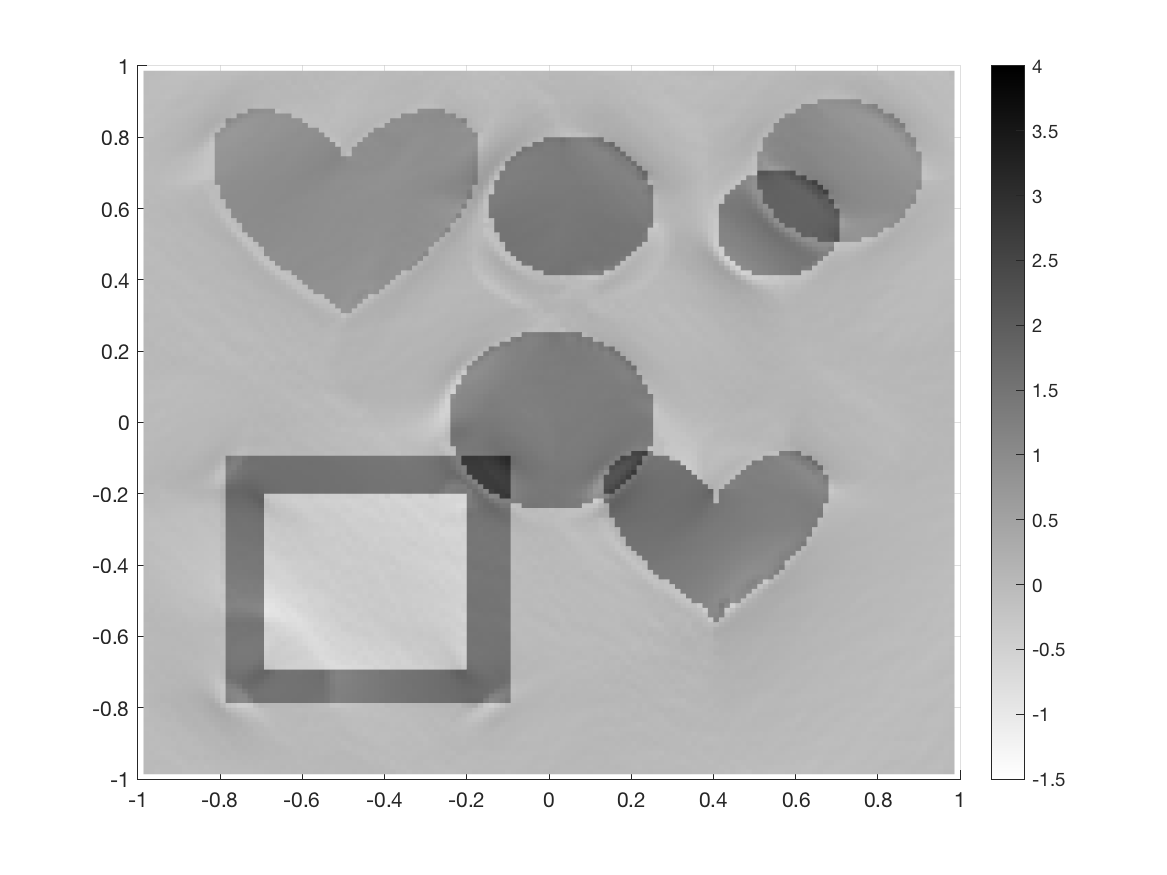}\label{mixed2_recon_lin}}

    \caption{Test Case 4- Another actual and reconstructed mixed phantom for various values of $\beta,\gamma,\delta$}
    \label{mixed2}
  \end{figure}
 
From Figure \ref{mixed2_recon3-01-01}, we observe the presence of a large number of artifacts. This is the case when the value of $\gamma = 0.01$ is small compared to the value of $\beta$. Increasing the value of $\gamma$ to 0.3 and, subsequently, to 0.5, reduces the number of artifacts as can be seen in Figures \ref{mixed2_recon3-3-01} and \ref{mixed2_recon3-5-01}. This shows that the $L^1$ regularization indeed leads to a nice contrast in the images, even when the sparsity level is low. Comparing the reconstruction obtained with the Picard algorithm, we see in Figure \ref{mixed2_recon_lin} the presence of a large number of artifacts and loss of contrast, especially in the hole and the inclusions.

\section{Conclusions}\label{sec:conclusions}

In this paper, we propose a new framework to facilitate high contrast and high resolution reconstructions in CDII. Our framework is based on formulating the CDII inverse problem as a PDE-constrained optimization problem. In this setup, we minimize an objective functional comprising of least square interior data fitting terms corresponding to two boundary voltage measurements, a $L^2-L^1$ penalization terms of the log-conductivity that helps promotes sparsity patterns, thus, improving contrast. Additionally we introduce a PM filtering term to sharpen the edges and obtain high resolution images.

We characterized the solution of the optimization problem through an optimality system that was solved using a proximal scheme, coupled with $H^1$ denoising to remove artifacts in the reconstructions.  We then demonstrated the effectiveness of our proposed scheme through several numerical experiments and compared our results with an existing Picard-type scheme. Our scheme facilitated reconstructions of a wide variety of conductivity patterns with superior contrast and resolution, even in the presence of noise in the data.

\section{Acknowledgments}
S. Roy was partly supported by the National Cancer Institute, National Institutes of Health, grant number: 1R21CA242933-01.

\end{document}